\documentclass[12pt]{amsart}

\usepackage[hmargin=2.75cm,bmargin=2.55cm,tmargin=3.05cm]{geometry}
\usepackage[usenames]{color}
\usepackage{subcaption}

\usepackage{mathtools}
\numberwithin{equation}{section}
\mathtoolsset{showonlyrefs,showmanualtags}
	
\usepackage{amsmath,amssymb,amsthm,amsfonts,mathrsfs,enumerate,url,tikz,graphicx}
\usetikzlibrary{calc}
\usepackage{hyperref}
\usepackage{float}

\usepackage[normalem]{ulem}
\usepackage[utf8]{inputenc}
\DeclareUnicodeCharacter{010D}{\vc}

\theoremstyle{plain}
\newtheorem{theorem}{Theorem}[section]
\newtheorem{lemma}[theorem]{Lemma}

\newtheorem{corollary}[theorem]{Corollary}

\theoremstyle{definition}
\newtheorem{definition}[theorem]{Definition}
\newtheorem{example}[theorem]{Example}
\newtheorem{conjecture}[theorem]{Conjecture}

\newtheorem{assumption}[theorem]{Assumption}
\newtheorem{remark}[theorem]{Remark}

\numberwithin{equation}{section}
\numberwithin{figure}{section}

\newcommand{\dx}{\; {\rm d}x}
\DeclareMathOperator{\interior}{int}
\DeclareMathOperator{\dist}{dist}

\DeclareMathOperator{\loc}{loc}
\DeclareMathOperator{\support}{supp}
\DeclareMathOperator{\acc}{acc}

\newcommand{\subsubset}{\subset\!\subset}

\newcommand{\Partition}{\mathcal{P}}
\newcommand{\Graph}{\mathcal{G}}
\newcommand{\HGraph}{\mathcal{H}}
\newcommand{\Tree}{\mathcal{T}}
\newcommand{\VertexSet}{\mathcal{V}}
\newcommand{\EdgeSet}{\mathcal{E}}

\newcommand{\doptenergy}[1]{\mathcal L_{#1}^D}
\newcommand{\denergy}[1]{\Lambda_{#1}}
\newcommand{\ess}{\sigma_{\mathsf{ess}}}
\newcommand{\infess}{\Sigma}
\newcommand{\infspec}{\lambda}
\newcommand{\RQ}{\mathcal{R}}

\newcommand{\N}{\mathbb{N}}
\newcommand{\R}{\mathbb{R}}
\newcommand{\C}{\mathbb{C}}

\DeclareMathOperator{\supp}{supp}

\title{Spectral minimal partitions of unbounded metric graphs} 



\author[M.~Hofmann]{Matthias Hofmann}
\author[J.~B.~Kennedy]{James B.~Kennedy}
\author[A.~Serio]{Andrea Serio}

\address{Department of Mathematics, Texas A{\&}M University, College Station, TX 77843-3368, United States of America}
\email{mhofmann@tamu.edu}

\address{Grupo de F\'isica Matem\'atica \textit{and} Departamento de Matem\'atica, Faculdade de Ci\^encias, Universidade de Lisboa, Campo Grande, Edif\'icio C6, P-1749-016 Lisboa, Portugal}
\email{jbkennedy@fc.ul.pt}

\address{Grupo de F\'isica Matem\'atica, Faculdade de Ci\^encias, Universidade de Lisboa, Campo Grande, Edif\'icio C6, P-1749-016 Lisboa, Portugal}
\email{aserio@fc.ul.pt}

\thanks{The authors wish to thank Aleksey Kostenko for some very helpful suggestions and comments regarding the literature and terminology around infinite metric graphs. J.B.K. and A.S. were supported by the Funda\c{c}\~ao para a Ci\^encia e a Tecnologia, Portugal, via the project NoDES: Nonlinear Dispersive and Elliptic Systems, reference PTDC/MAT-PUR/1788/2020 (J.B.K.), and grant UIDB/00208/2020 (both authors).}

\begin{document}
\begin{abstract}
We investigate the existence or non-existence of spectral minimal partitions of unbounded metric graphs, where the operator applied to each of the partition elements is a Schrödinger operator of the form $-\Delta + V$ with suitable (electric) potential $V$, which is taken as a fixed, underlying function on the whole graph.

We show that there is a strong link between spectral minimal partitions and infimal partition energies on the one hand, and the infimum $\infess$ of the essential spectrum of the corresponding Schrödinger operator on the whole graph on the other. Namely, we show that for any $k\in\N$, the infimal energy among all admissible $k$-partitions is bounded from above by $\infess$, and if it is strictly below $\infess$, then a spectral minimal $k$-partition exists. We illustrate our results with several examples of existence and non-existence of minimal partitions of unbounded and infinite graphs, with and without potentials.

The nature of the proofs, a key ingredient of which is a version of the characterization of the infimum of the essential spectrum known as Persson's theorem for quantum graphs, strongly suggests that corresponding results should hold for Schrödinger operator-based partitions of unbounded domains in Euclidean space.
\end{abstract}

\maketitle

\section{Introduction}
\label{sec:intro}

Our goal is to investigate the existence and non-existence of spectral minimal partitions (SMPs) of infinite graphs, and the behavior of the corresponding infimal spectral energies.

For a given $k\geq 1$, the problem of finding a spectral minimal $k$-partition typically involves minimizing, over all $k$-partitions of a given object such as a bounded domain, manifold or compact graph, a functional based on a Laplacian-type eigenvalue of each of the partition elements. They were introduced on domains around 20 years ago (see the seminal paper \cite{CoTeVe05}), and have been intensively studied in that setting ever since, often as a counterpart to the nodal partitions of the domain induced by its Laplacian eigenfunctions; the sequence of energies of the optimal partitions (i.e.\ the value of the corresponding functional on them) may then be compared with the sequence of eigenvalues of the domain. We refer to the survey \cite{BNHe17} for more details and references.

On metric graphs the study of SMPs is more recent; they were probably first introduced in \cite{BBRS12} in order to give insight into nodal partitions of Laplace eigenfunctions of the whole graph. The related study of the \emph{nodal deficit} (and \emph{nodal statistics}) of an eigenfunction, that is, the extent to which an eigenfunction of the $k$th eigenvalue fails to have $k$ nodal domains, has proved to be a rich and fruitful line of investigation \cite{ABB18,ABB22,Berkolaiko08,BCCM22}. It has also started to yield insights into domains; see in particular \cite{BCHPS22,BeCoMa19}. A general existence theory of SMPs of compact graphs was laid out only very recently in \cite{KKLM21}; further properties of SMPs and the associated minimal energies in this setting were subsequently studied in \cite{HoKe21,HKMP21}.

The principal goal of the current work is to investigate the case of \emph{unbounded} metric graphs; to the best of our knowledge this is the first time that SMPs have been studied in an unbounded setting. The usual self-adjoint realizations of the Laplacian may have essential spectrum on the whole graph, and the same may be true of the elements of a graph partition. Indeed, the presence or absence of essential spectrum, as well as eigenvalues below it, can depend intricately on metric and topological features of the graph \cite{KoNi19,So04}. The first question thus becomes exactly how to define the energy functionals; but the presence of essential spectrum also fundamentally alters the nature of SMPs and their corresponding minimal energies, as we shall see shortly.

Partially for this reason, we will also enlarge the class of operators we consider to include general Schrödinger operators of the form $-\Delta + V$, where the potential $V$ is locally integrable. On the other hand, we will restrict ourselves to the model case of Dirichlet conditions at the boundary between partition elements, as is usual in the case of domains but less general than the theory developed for compact metric graphs, together with standard (Neumann--Kirchhoff) conditions at all other vertices.

To formulate our results, we first need to introduce some notation and definitions; for more details see Section~\ref{sec:preliminaries}. Throughout, $\Graph = (\VertexSet, \EdgeSet)$ will be a metric graph consisting of a countable (possibly finite) vertex set $\VertexSet$ and a countable (possibly finite) edge set $\EdgeSet$, each edge $e$ being identified either with a compact interval $[0,\ell_e]$ or a half-line $[0,\infty)$; for our graph constructions we will take the formalism of \cite{Mu19} and \cite{KKLM21}. We consider a fixed potential $V\in L^1_{\loc}(\Graph)$, for simplicity assumed nonnegative, and take the following standing assumptions:

\begin{assumption}
\label{ass:main-assumption}
The metric graph $\Graph$ is connected and satisfies the finite ball condition: every ball of finite radius in $\Graph$ intersects a finite number of edges of $\Graph$. 
The potential $V: \Graph \to \R$ is in $L^1_{\textrm{loc}} (\Graph)$ with $V \geq 0$ almost everywhere.
\end{assumption}

This includes as a special case all graphs whose edge lengths are uniformly bounded away from zero, and in particular all finite graphs in the sense of \cite[Definition~1.3.4]{BeKu13}; such graphs were considered for instance in \cite{Ho19,KoMuNi19,KoNi19}.

For any subgraph $\HGraph \subset \Graph$ we define $\partial\HGraph$ to be its topological boundary in the metric space $\Graph$ (and \emph{not} the set of its degree one vertices). By a $k$-\emph{partition} $\Partition  = (\Graph_1,\ldots,\Graph_k)$ of $\Graph$ we understand a collection of $k \in \N$ \emph{connected} subgraphs $\Graph_1,\Graph_2,\ldots,\Graph_k$, such that, for all $i,j$, $\Graph_i \cap \Graph_j \subset \partial \Graph_i \cap \partial \Graph_j$. The partition elements, that is, the subgraphs $\Graph_i$, which as in \cite{KKLM21} we also call \emph{clusters}, may be bounded or unbounded,  and we do \emph{not} require that their union equal $\Graph$.

Given a subgraph $\HGraph \subset \Graph$, which may be equal to $\Graph$, we define the spaces $L^2(\HGraph)$ and $C(\overline{\HGraph})$ of square integrable and continuous functions in the usual way. Since $V$ will be fixed throughout, the Sobolev space $H^1(\HGraph) \subset C(\overline{\HGraph})$ will denote the space of functions for which the norm
\begin{displaymath}
	\|f\|_{H^1(\HGraph)}^2 := \|f'\|_{L^2(\HGraph)}^2 + \|(V|_\HGraph+1)^{1/2}f\|_{L^2(\HGraph)}^2
\end{displaymath}
is finite, and $H^1_0(\HGraph)$ is the space of $H^1(\HGraph)$-functions vanishing at all points in $\partial\HGraph$.

For $\HGraph$ connected, or at most a finite union of closed connected subgraphs of $\Graph$, we will consider the model Schrödinger operator $A_\HGraph : D(A_\HGraph) \subset L^2 (\HGraph) \to L^2(\HGraph)$ given by $-\Delta+ V|_\HGraph$ and satisfying Dirichlet conditions on $\partial\HGraph$ and standard (continuity and Kirchhoff) conditions at all vertices in the interior of $\HGraph$. Equivalently, this is the operator on $L^2(\HGraph)$ associated with the form
\begin{equation}\label{eq:form-intro}   
a_{\HGraph}(f,g) = \int_{\HGraph} f' \cdot \overline g' + V f \overline g, \, \mathrm dx, \qquad f,g \in H_{0}^1(\HGraph) .
\end{equation}   
(In practice we will write $V$ in place of $V|_\HGraph$.)

Treating the underlying potential $V$ as fixed and the subgraph $\HGraph$ as variable, we denote by $\sigma (\HGraph) \subset \C$ the spectrum of this operator, by $\infspec (\HGraph)$ the infimum of its spectrum, and by $\infess (\HGraph) \in \R \cup \{\infty\}$ the infinimum of its essential spectrum, $\infess (\HGraph) \geq \infspec (\HGraph)$ (see Section~\ref{sec:spectral-properties}).

Given $k \in \N$ we will always consider the $\infty$-energy functional associated with $A_\HGraph$ acting on $k$-partitions, $\denergy{k}: \Partition \mapsto \R$; that is, if $\Partition = (\Graph_1,\ldots,\Graph_k)$, then we set
\begin{equation}
\label{eq:denergy-def}
	\denergy{k} (\Partition) := \max \{ \infspec (\Graph_1),\ldots, \infspec (\Graph_k) \} \in [0,\infty).
\end{equation}
For $k \in \N$ the corresponding spectral minimization problem involves studying the quantity
\begin{equation}
\label{eq:spec-min}
	\doptenergy{k} (\Graph) = \inf \{ \denergy{k} (\Partition) : \Partition \text{ is a $k$-partition of $\Graph$} \},
\end{equation}
as was studied on compact graphs in \cite{KKLM21}, and as has been done on domains very extensively over the last two decades \cite{BNHe17}.

Our main results link this optimal energy $\doptenergy{k} (\Graph)$, the existence of a minimizer, and the infimum of the essential spectrum of the operator $-\Delta + V$ on the whole graph $\Graph$.

\begin{theorem}
\label{thm:less_than_ess}
For any $k \geq 1$, we have
\begin{equation}
\label{eq:less_or_equal_than_ess}
        \doptenergy{k}(\Graph)\leq \infess(\Graph).
\end{equation}
\end{theorem}

\begin{theorem}
\label{thm:existenceprinciple}
For any $k \geq 1$, if
\begin{equation}
\label{eq:less_than_ess}
        \doptenergy{k}(\Graph) < \infess(\Graph),
\end{equation}
then there exists a $k$-partition $\Partition = (\Graph_1,\ldots,\Graph_k)$ realizing $\doptenergy{k} (\Graph)$. Moreover, for any $i=1,\ldots,k$, $\lambda (\Graph_i)$ is an isolated eigenvalue with a corresponding ground state, i.e.\ positive eigenfunction, on $\Graph_i$.
\end{theorem}

It follows immediately from the preceding two theorems that:

\begin{corollary}
\label{cor:ifandonlyif}
For any $k \geq 1$, there exists a $k$-partition realizing $\doptenergy{k} (\Graph)$ if and only if there exists a $k$-partition $\Partition$ of $\Graph$ such that
\begin{equation}
\label{eq:if_and_only_if}
    \denergy{k}(\Partition) \leq \infess (\Graph).
\end{equation}
\end{corollary}

\begin{remark}
\label{rem:groundstates}
Examples show that if $\doptenergy{k}(\Graph) = \infess (\Graph)$, then minimizing partitions may or may not exist, and even if they do, then their clusters may or may not have ground states (Section~\ref{sec:examples}).
\end{remark}

These principles strongly recall the minimax theorem for the eigenvalues of a self-adjoint operator $A$ on an unbounded domain $\Omega \subset \R^d$ (see \cite[Theorem~4.14]{Te14}), i.e., if there exist $k$ linearly independent test functions with energy below the infimum of its essential spectrum $\infess$, then there must be $k$ eigenvalues below $\infess$.

The basic principle behind Theorem~\ref{thm:less_than_ess} is a characterization of the infimum of the essential spectrum $\infess (\Graph)$ of a graph (or any subgraph thereof) in terms of the infimum of the spectrum of the exterior of expanding balls: if we fix any root point $0 \in \Graph$ in the graph, then
\begin{equation}
\label{eq:persson-intro}
	\infess (\Graph) = \lim_{r \to \infty} \infspec (\Graph \setminus B_r(0)),
\end{equation}
where $B_r(0) = \{ x \in \Graph : \dist_\Graph (x,0) \leq r\}$. This result is valid for a wide class of operators including Schr\"odinger operators with magnetic potential; its version in Euclidean space is often called \emph{Persson's theorem} in the mathematical physics literature (see \cite[§14.4]{HiSi96} or \cite{Pe60}). Our version, Theorem~\ref{thm:persson}, is based on \cite{Ho21} (see also \cite{AkPa16}). Together with domain monotonicity results, it implies that for any $\lambda > \infess (\Graph)$ we can find an infinite sequence of nested annuli
\begin{displaymath}
	A_{R_{n},R_{n+1}} (0) := \{ x \in \Graph: R_{n} \leq \dist_\Graph (x,0) \leq R_{n+1} \}, \qquad n \in \N,
\end{displaymath}
$R_1 < R_2 < R_3 < \ldots$, such that $\infspec (A_{R_{n},R_{n+1}}) \leq \lambda$. Their existence immediately implies Theorem~\ref{thm:less_than_ess}, since for any $\lambda > \infess (\Graph)$ and any $k \in \N$ we can exhibit a $k$-partition based on these annuli whose spectral energy is no larger than $\lambda$.

The existence result of Theorem~\ref{thm:existenceprinciple}, on the other hand, requires (among other things) a very careful diagonal argument, together with existence results for compact graphs, to extract a convergent subsequence in a suitable sense from a minimizing sequence of partitions. It is here that we will make heavy use of the metric graph setting.

Nevertheless, since the characterization \eqref{eq:persson-intro} is certainly valid for Schrödinger operators in the Euclidean setting, the main results should continue to hold on domains:

\begin{conjecture}
\label{conj:domains}
Let $\Omega \subset \R^d$, $d \geq 2$, be a domain, possibly equal to $\R^d$, and let $V: \Omega \to \R$ be a sufficiently smooth potential, bounded from below. Denote the infimum of the essential spectrum of the Schrödinger operator $-\Delta+V$, with Dirichlet conditions on $\partial\Omega$ if $\Omega \neq \R^d$, by $\infess (\Omega)$. Given a $k$-partition $(\Omega_i)_{i=1}^k$ of $\Omega$ into open, connected and pairwise disjoint sets $\Omega_i$, denote by $\infspec(\Omega_i)$ the infimum of the spectrum of $-\Delta+V$ on $\Omega_i$ with Dirichlet boundary conditions, and set
\begin{displaymath}
	\doptenergy{k} (\Omega) = \inf \max_{i=1,\ldots,k} \infspec (\Omega_i),
\end{displaymath}
where the infimum is taken over all such $k$-partitions (see \cite[Section~10.1]{BNHe17}). Then:
\begin{enumerate}
\item $\doptenergy{k} (\Omega) \leq \infess (\Omega)$ for all $k \geq 1$;
\item If $\doptenergy{k} (\Omega) < \infess (\Omega)$ for some $k \geq 1$, then there exists a $k$-partition realizing $\doptenergy{k} (\Omega)$;
\item In particular, if there exists a $k$-partition whose spectral energy is less than or equal to $\infess (\Omega)$, then there exists a $k$-partition realizing $\doptenergy{k} (\Omega)$.
\end{enumerate}
\end{conjecture}

This paper is organized as follows. In Section~\ref{sec:preliminaries} we will discuss our assumptions and notation around metric graphs in more detail, construct the operators of interest via the associated sesquilinear forms, and discuss a number of basic spectral properties. Section~\ref{sec:spectral} is devoted to the behavior of the spectral quantities $\infspec (\HGraph)$ and $\infess (\HGraph)$ in dependence on the subgraph $\HGraph$; in particular, we will formulate auxiliary results on continuity and monotonicity with respect to the subgraph. The proofs of Theorems~\ref{thm:less_than_ess} and~\ref{thm:existenceprinciple} will then be given in Section~\ref{sec:proofs}.

In Section~\ref{sec:examples} we will conclude with a number of examples which illustrate what can happen when $\doptenergy{k} (\Graph) = \infess (\Graph)$, as mentioned in Remark~\ref{rem:groundstates}: both existence and non-existence of a spectral minimal $k$-partition is possible (Example~\ref{ex:main}); it is possible that $\doptenergy{k} (\Graph) = \infess (\Graph) > 0$ even if $V=0$ (Example~\ref{ex:v0ls}); and if a spectral minimal $k$-partition exists then the partition elements may or may not admit ground states (Example~\ref{ex:startree}).

A number of the proofs from Section~\ref{sec:spectral} will be deferred to an appendix. There, we study spectral convergence, for our type of Schrödinger operators, of a sequence of (possibly unbounded) subgraphs $\HGraph_n \subset \Graph$ to a limit subgraph $\HGraph$. This includes lower semicontinuity of $\infspec (A_{R_{n},R_{n+1}})$ with respect to the radii $R_n$ and $R_{n+1}$, although the results we formulate are a bit more general. To the best of our knowledge such convergence results are new in the setting of non-compact graphs, and may be of some independent interest.

\section{Infinite graphs and Schrödinger operators}
\label{sec:preliminaries}

\subsection{On metric graphs}

Throughout, we will assume $\Graph = (\VertexSet,\EdgeSet)$ to be a fixed metric graph satisfying Assumption~\ref{ass:main-assumption} without further comment.
Where necessary we will use the formalism of \cite{Mu19} for metric graphs. We assume without loss of generality that every edge $e \in \EdgeSet$ may be identified with a compact interval $[0,\ell_e]$; this can always be arranged by inserting degree two (``dummy'') vertices as necessary on \textit{a priori} unbounded edges. As always for Schrödinger-type operators, the resulting implicit choice of orientation of the edges will be irrelevant for our purposes. We do \emph{not} assume the set of edge lengths $\{\ell_e : e \in \EdgeSet\}$ to be bounded from above, or from below away from zero.

For any subset $\HGraph$ of $\Graph$, we denote by $\partial\HGraph$ its topological boundary in $\Graph$, which without loss of generality we assume to consist purely of (possibly dummy) vertices of $\Graph$, and by $\overline\HGraph$ its closure. Since $\Graph$ is a complete metric space, any closed subset of $\Graph$ is also a complete metric space for the induced metric. We identify any closed subset $\HGraph$ of $\Graph$ having a finite number of connected components as a \emph{subgraph} of $\Graph$; in particular, for us subgraphs are closed. In this context, we denote by $\EdgeSet(\HGraph)$ the corresponding edge set of $\HGraph$ and $\VertexSet(\HGraph) $ the corresponding vertex set of $\HGraph$.

We call a subgraph $\HGraph$ of $\Graph$ \emph{compact} if it is compact as a metric space for the induced metric. For our definition of subgraph, this is equivalent to $\HGraph$ intersecting a finite number of edges of $\Graph$. In particular, $\Graph$ is itself compact if and only if it has a finite number of edges, each of finite length; cf.\ \cite[Definition~1.3.4]{BeKu13}.

\begin{definition}
\label{def:balls}
Fix $x \in \Graph$, and denote by $\dist_\Graph(x,y)$ the canonical Euclidean distance in $\Graph$ between $x,y \in \Graph$.
\begin{enumerate}
\item For $r>0$ we denote by
\begin{displaymath}
    B_r(x) := \{ y \in \Graph: \dist_\Graph (x,y) \leq r\}
\end{displaymath}
the \emph{closed} ball of radius $r$ centered at $x$.
\item For $r_2 \geq r_1 > 0$ we denote by
\begin{displaymath}
    A_{r_1,r_2}(x) := \acc \{y \in \Graph: r_1 \leq \dist_\Graph (x,y) \leq r_2\}
\end{displaymath}
the closed annulus of inner radius $r_1$ and outer radius $r_2$, with any isolated points removed.
\end{enumerate}
\end{definition}

Note that $B_r(x)$ is necessarily always connected, but $A_{r_1,r_2}(x)$ in general will not be.

When considering subgraphs $\HGraph$ of $\Graph$, we will always take $\dist = \dist_\Graph$ to be the distance function on $\Graph$ restricted to $\HGraph$, and \emph{not} the intrinsic distance function on $\HGraph$ as a graph.

\subsection{Function spaces, sesquilinear forms and associated operators}

Throughout, we fix a subgraph $\emptyset \neq \HGraph \subset \Graph$. This subgraph $\HGraph$ may be equal to $\Graph$. We also fix a potential $V: \Graph \to \R$ satisfying Assumption~\ref{ass:main-assumption}.

We define the function spaces $L^p (\HGraph)$, $p \in [1,\infty)$, in the usual way, as the set of all edgewise $L^p$-functions such that the $p$-series of all edgewise $L^p$-norms is finite: $f \in L^p (\HGraph)$, if
\begin{displaymath}
	\|f\|_p := \left(\sum_{e \in \EdgeSet (\HGraph)} \int_e |f|_e|^p\,\mathrm dx\right)^{1/p} < \infty;
\end{displaymath}
when $p=\infty$ we require $\|f\|_\infty := \sup_{e \in \EdgeSet (\HGraph)} \|f|_e\|_\infty$ to be finite.  We say $f \in L^p_{\loc} (\HGraph)$ is a \emph{locally finite $L^p$-function} if $f|_{B_r(x) \cap \HGraph} \in L^p (B_r(x) \cap \HGraph)$ for \emph{all} $r > 0$ and $x \in \HGraph$. The space of continuous functions $C(\HGraph)$ is defined in the usual way, with respect to the natural metric on $\HGraph$.

Define the Sobolev space
\begin{equation}\label{eq:h1}
    H^1(\Graph) := \left \{ u\in L^2(\Graph) : \int_{\Graph} |u'|^2 +  V |u|^2 \, \mathrm dx <\infty \right \}.
\end{equation}
We similarly define 
\begin{equation}\label{eq:h10}
H_0^1(\HGraph) := \{ u\in H^1(\Graph),\, \operatorname{supp} u \subset \HGraph\}
\end{equation}
for a subgraph $\HGraph$ of $\Graph$; in a slight abuse of notation we will not distinguish between functions on $\HGraph$ which vanish on $\partial\HGraph$ and their extension by zero to functions on $\Graph$. The space $H^1(\HGraph)$ is equipped with the inner product
\begin{displaymath}
    \langle f,g\rangle_{H^1} = \int_\HGraph f'\cdot \overline{g}'
    + (V|_\HGraph + 1)f\overline{g}\,\mathrm d\mu.
\end{displaymath}
We then define the sesquilinear form $a_\HGraph : H^1_0 (\HGraph) \times H^1_0 (\HGraph) \to \R$ by
\begin{equation}
\label{eq:form}
	a_\HGraph (f,g) = \int_\HGraph f'\cdot \overline{g}' + V|_{\HGraph} f\overline{g}\,\mathrm d \mu, \qquad f,g \in H^1_0 (\HGraph),
\end{equation}
 so that 
 \begin{equation}\label{eq:h1-norm}
     \|f\|_{H^1}^2 = a_\HGraph (f,f) + \|f\|_2^2.
 \end{equation} 
 We will write $a_\HGraph (f)$ in place of $a_\HGraph (f,f)$, and  we will often write $V$ in place of $V|_{\HGraph}$.
\begin{lemma}\label{lem:DensityClosed}
 Under Assumption~\ref{ass:main-assumption}, $H^1(\Graph)$ is complete and
$
     H_0^1(\HGraph)
$ 
is a closed subspace of $H^1(\Graph)$.
\end{lemma}

\begin{proof}
Consider a Cauchy sequence in $(u_n) \in H^1(\Graph)$, then $(u_n), (V^{1/2} u_n)$ are also Cauchy in $L^2(\Graph)$, and $(u_n')$ is Cauchy in $L^2(\Graph)$. Since these spaces are complete, there exist $u\in L^2(\Graph)$ and $v\in L^2(\Graph)$ such that $u_n \to u$, $V^{1/2}u_n \to V^{1/2} u$ and $u_n' \to v$ in $L^2 (\Graph)$. Moreover, for all $e \in \EdgeSet$, $u_n |_e$ admits some limit $u_e \in H^1(e)$ on that edge. The only possibility is that $v=u'$, and $u_n \to u$ in $H^1(\Graph)$.

Now suppose $\HGraph$ is a subgraph of $\Graph$ and $u_n\to u$ is a convergent sequence of $u_n\in H_0^1(\HGraph)$ in $H^1(\Graph)$. Then for every edge there exists a subsequence such that the sequence converges pointwise almost everywhere to $u$. It follows that $\operatorname{supp} u \subset \mathcal H$.  Hence $u\in H_0^1(\HGraph)$.
\end{proof}

\begin{lemma}
\label{lem:form}
Let $\HGraph$ be a subgraph of $\Graph$. Then under Assumption~\ref{ass:main-assumption}, $H^1_0 (\HGraph)$ equipped with the norm \eqref{eq:h1-norm} is a Hilbert space which is densely imbedded in $L^2(\HGraph)$. Moreover, $a_\HGraph$ is a continuous, symmetric quadratic form on $H^1_0 (\HGraph)$, and $a_\HGraph + \langle \cdot, \cdot \rangle_{L^2(\HGraph)}$ is coercive on $H^1_0 (\HGraph)$.
\end{lemma}

Due to the symmetry of the form, we will henceforth \emph{only work with real-valued functions} and not complex extensions.

\begin{proof}
The form is closed due to Lemma~\ref{lem:DensityClosed}. For the density, consider the set
\begin{multline*}
	\{ f \in C(\HGraph): f|_e \in C^\infty (e) \text{ for all } e \in \EdgeSet(\HGraph),\,\,
	\support f \text{ is compactly contained in } \HGraph,\\ \text{and } f \equiv 0 \text{ in a neighborhood of each vertex } v \in \VertexSet(\HGraph)\}
\end{multline*}
contained in $H^1_0 (\HGraph)$. Since on any open interval $\mathcal I$ the space of smooth functions compactly supported in $\mathcal I$ is dense in $L^2 (\mathcal I)$, and $L^2(\HGraph)$ is a direct sum of spaces of the form $L^2(\mathcal I)$, an elementary approximation argument yields the density of $H^1_0 (\HGraph)$ in $L^2(\HGraph)$. That $a_\HGraph$ is continuous and $a_\HGraph + \langle \cdot, \cdot \rangle_{L^2(\HGraph)}$ is coercive is immediate, since $a_\HGraph (f) + \|f\|_{L^2(\HGraph)}^2$ coincides with the square of the norm, and $a_\HGraph$ is obviously symmetric.
\end{proof}

It follows that the associated operator, which we will denote by $A_\HGraph := -\Delta + V|_\HGraph$, is self-adjoint and bounded from below on $L^2(\HGraph)$; a routine calculation shows that $A_\HGraph$ satisfies a Dirichlet condition at all vertices of $\partial\HGraph$ and standard (continuity plus Kirchhoff) conditions at all other vertices. Whenever $\Graph$ is complete as a metric space and $V \equiv 0$, the corresponding operator is the Laplace operator as considered in \cite{EKMN18,KoNi19}, and the self-adjointness is due to \cite[Corollary~4.9]{EKMN18}. This is the case for the regular tree graphs considered in \cite{SoSo02,So04}, which we will revisit in Section~\ref{sec:examples}.

\begin{theorem}\label{thm:self-adjoint}
The self-adjoint operator $A$ associated with the closed, semibounded symmetric form $a_{\HGraph}(\cdot, \cdot)$ on $H^1_0(\HGraph)$ can be characterized via
\begin{equation}\label{eq:operator}
    \begin{aligned}
     D(A_{\HGraph}) &= \{ f\in H_0^1(\HGraph): -f'' + V f \in L^2(\mathcal H), \\
     &\qquad \qquad \qquad \sum_{e\succ v}  \frac{\partial}{\partial \nu} f_e(v) =0 \text{ for all } v\in \VertexSet(H) \setminus \partial \HGraph \}\\
    A_{\HGraph} f|_e &=  (- f'' + V f)|_e \qquad \text{for all } e \in
    \EdgeSet(\HGraph)
     \end{aligned}
\end{equation}
\end{theorem}

\begin{proof}
One easily verifies that $A_{\HGraph}\subset A$. Indeed, suppose $f\in D(A_{\HGraph})$ and $g\in H^1(\mathcal G)$ with compact support, then by definition of $D(A_\HGraph)$,
\begin{equation}
    a_{\HGraph}(f, g) = \int_{\mathcal H} - f'' g + V fg \, \mathrm d\mu = \langle A_{\HGraph} f, g \rangle_{L^2},
\end{equation}
and due to the density of compactly supported functions as shown in Lemma~\ref{lem:form} we infer $D(A_\HGraph)\subset D(A)$ and $A\big |_{\HGraph} = A_\HGraph$. For the other direction, by restricting to edgewise supported test functions, one sees that edgewise $Af$ satisfies
\begin{equation}
    Af =  - f'' + V f.
\end{equation}
Taking test functions $g$ locally supported in a neighborhood of a single vertex $v$ leads to 
\begin{equation}
    \left (\sum_{e\succ v} - \frac{\mathcal \partial}{\partial \nu} f_e(v) \right ) g(v) = \langle Af, g\rangle - \langle Ag, f \rangle =0
\end{equation}
for all $v\in \VertexSet(\HGraph)\setminus \partial \HGraph$. This establishes that $A\subset A_{\HGraph}$, and so there is equality.  
\end{proof}

\subsection{Spectral properties of infinite quantum graphs}
\label{sec:spectral-properties}
Denote by $\RQ_\HGraph$ the Rayleigh quotient 
\begin{equation}
\label{eq:rayleigh}
 	\RQ_{\HGraph} [f] = \frac{\int_\HGraph |f'|^2 + V|f|^2\,\dx}{\int_\HGraph |f|^2\,\dx}, \qquad 0 \neq f \in H^1_0 (\HGraph),
\end{equation}
associated with the operator $A_\HGraph$. We consider the infimum of the spectrum of $A_\HGraph$, given by
\begin{equation}
\label{eq:infspec}
	\infspec (\HGraph) = \inf_{0 \neq f \in H^1_0 (\HGraph)} \RQ_\HGraph [f] \in [0,\infty),
\end{equation}
as well as the infimum of the essential spectrum, $\infess (\HGraph) \in [\infspec (\HGraph),\infty]$. If there are any isolated eigenvalues below $\infess (\HGraph)$, we will denote them by
\begin{displaymath}
	0 \leq \infspec (\HGraph) \equiv \lambda_1 (\HGraph) \leq \lambda_2 (\HGraph) \leq \ldots < \infess (\HGraph),
\end{displaymath}
repeated according to their necessarily finite multiplicities; these are given by the usual min-max and max-min characterizations.

We finish with the characterization, mentioned in the introduction, of $\infess (\cdot)$ in terms of \emph{Persson theory}, inspired by \cite[Theorem~2.4.22]{Ho21} (but see also \cite{AkPa16}). This is motivated by a well-established theory on domains in $\R^d$; see \cite{Pe60} or \cite[§14.4]{HiSi96}. To this end we first need a kind of cut-off formula for the form $a_\HGraph$ (sometimes called the \emph{IMS localization formula}, named after Ismigilov, Morgan and Simon and also I.M. Sigal as explained in \cite[§2]{Si85}). This formula will also play a key role in the appendix.

\begin{lemma}
\label{lem:ims}
Given a subgraph $\HGraph$ of $\Graph$, the sesquilinear form $a_\HGraph(u,v)$ satisfies
\begin{equation}
\label{eq:ims}
	a_\HGraph (\phi f, \phi g) = \frac{1}{2}\Big[a_\HGraph (\phi^2 f, g) + a_\HGraph (f,\phi^2 g)\Big] + \langle |\phi'|^2 f,g\rangle_{L^2(\HGraph)}
\end{equation}
for all $f,g \in H^1_0 (\HGraph)$ and all $\phi \in C(\HGraph) \cap L^{\infty} (\HGraph)$ with $\phi' \in L^\infty(\HGraph)$.
\end{lemma}
\begin{proof}
Suppose $\phi$ is a continuous function satisfying $\phi, \phi' \in L^\infty (\HGraph)$ and $f\in H_0^1(\HGraph)$. Then $\supp(\phi f)\subset \HGraph$ and
\begin{equation}
    (\phi f)' = \phi' f  + \phi f'
\end{equation}
edgewise. This relation implies that $\phi f\in H_0^1(\HGraph)$. That $a_\HGraph$ satisfies \eqref{eq:ims} is now reduced to a formal (and standard) calculation, which can be found in \cite[Proposition~2.8]{Si85}.
\end{proof}

\begin{theorem}[Persson's theorem]
\label{thm:persson}
Suppose $\Graph$ and $V$ satisfy Assumption~\ref{ass:main-assumption} and $\HGraph$ is a subgraph of $\Graph$. Then $A_\HGraph$ satisfies
\begin{equation}
\label{eq:ims-characterization}
	\infess (\HGraph) =\sup_{\mathcal K \subsubset\HGraph}\,\,\inf_{\substack{0\neq f \in H^1_0(\HGraph)\\
	\support f \cap \mathcal K = \emptyset}}\RQ_\HGraph[f].
\end{equation}
\end{theorem}

\begin{proof}
Due to the compact imbedding of $H^1(\HGraph)$ into $L^2_{\text{loc}}(\HGraph)$ and the continuous imbedding of $D(A_\HGraph)$ into $L^2(\HGraph)$, we infer that $D(A_\HGraph)$ locally compactly imbeds into $L^2(\HGraph)$. The statement then follows immediately from \cite[Theorem~2.4.22]{Ho21}, using Lemma~\ref{lem:ims}.
\end{proof}

\section{Continuity and monotonicity properties of the infimum of the spectrum}
\label{sec:spectral}

In this section we will prove the background continuity and monotonicity properties of $\infspec (\cdot)$, as well as related properties of $\infess (\cdot)$, with respect to expanding or contracting balls. Throughout, we will continue to suppose that $\Graph$ and $V$ satisfy Assumption~\ref{ass:main-assumption}.

\begin{lemma}
\label{lem:monotonicity}
Suppose that $\Graph_2 \subset \Graph_1$ are any nested subgraphs of $\Graph$ (where we explicitly allow $\Graph_1 = \Graph$). Then
\begin{enumerate}
\item $\infspec (\Graph_2) \geq \infspec (\Graph_1)$;
\item if $\Graph_1$ has discrete spectrum, then so too does $\Graph_2$, and $\lambda_k(\Graph_2) \geq \lambda_k (\Graph_1)$ for all $k \geq 1$.
\item if $\Graph_2$ has essential spectrum, then so too does $\Graph_1$, and $\infess (\Graph_2) \geq \infess (\Graph_1)$.
\end{enumerate}
\end{lemma}

\begin{proof}
For (1) we use that any element $u \in H^1_0 (\Graph_2)$ may be extended by zero to obtain an element of $H^1_0(\Graph_1)$ with the same Rayleigh quotient; for (2) we additionally use that $\Graph_i$ has discrete spectrum if and only if the imbedding of $H^1_0 (\Graph_i)$ into $L^2(\Graph_i)$ is compact. (3) is a direct consequence of the same argument as in (1) together with Theorem~\ref{thm:persson}.
\end{proof}

We fix any subgraph $\HGraph$ of $\Graph$ and any point $x \in \Graph$, not necessarily in $\HGraph$. For each $R>0$, we consider the subgraphs $\HGraph \cap B_R(x)$ and $\HGraph \setminus B_R(x)$. As subgraphs, these are always taken to be closed, but we define them so as not to have isolated points; thus, for example, if $\dist_\Graph (v,x) = R$ but $\dist_\Graph (y,x) > R$ for all $y \in \HGraph \setminus \{v\}$, then we take $\HGraph \cap B_R(x)$ to be the empty set instead of $\{v\}$.

\begin{lemma}
\label{lem:expanding-balls-at-0}
Under the above assumptions, let $x \in \Graph$ be fixed but arbitrary, and let $R_0 \geq 0$ be such that $\HGraph \cap B_R (x)$ has positive volume if and only if $R > R_0$. Then the operator $A_{\HGraph \cap B_R (x)}$ has discrete spectrum, and the mapping
\begin{equation}
\label{eq:expanding-ball}
	R \mapsto \infspec (\HGraph \cap B_R(x)) \in \R \cup \{\infty\}
\end{equation}
is a lower semicontinuous and monotonically decreasing function of $R \in (R_0,\infty)$, which satisfies
\begin{displaymath}
	\lim_{R \to R_0} \infspec (\HGraph \cap B_R(x)) = \infty, \qquad \lim_{R\to \infty} \infspec (\HGraph \cap B_R(x)) = \infspec (\HGraph).
\end{displaymath}
\end{lemma}

\begin{proof}
First note that $\HGraph \cap B_R(x)$ is always compact, for all $R>0$ and all $x \in \Graph$, as the intersection of a closed and a compact set; in particular, since $V|_{\HGraph \cap B_R (x)} \in L^1 (\HGraph \cap B_R (x)$, $A_{\HGraph \cap B_{R(x)}}$ has discrete spectrum (see \cite[Section~2]{HKMP21pleijel}).

Monotonicity follows directly from Lemma~\ref{lem:monotonicity}(1); semicontinuity will be proved in Appendix~\ref{sec:appendix-continuity} (see Lemma~\ref{lem:domain-continuity}). For the first limit, we note that, since $R \mapsto |\HGraph \cap B_R(x)|$ is clearly a continuous function, in particular
\begin{displaymath}
    |\HGraph \cap B_R (x)| \to 0
\end{displaymath}
as $R \to R_0$. Since $V \geq 0$, a direct comparison of Rayleigh quotients together with Nicaise' inequality \cite[Th\'eor\`eme~3.1]{Ni87} applied to the compact graph $\HGraph \cap B_R(x)$ implies that
\begin{displaymath}
    \infspec (\HGraph \cap B_R(x)) \geq \frac{\pi^2}{4|\HGraph \cap B_R(x)|^2} \to \infty.
\end{displaymath}
For the second limit, there is nothing to prove if $\HGraph$ is bounded, so we assume it is unbounded. It clearly suffices to prove that $\infspec (\HGraph) \geq \limsup_{R\to \infty} \infspec (\HGraph \cap B_R(x))$; to this end, by the variational characterization \eqref{eq:infspec} and monotonicity, we only have to prove that for any $u \in H^1_0 (\HGraph)$ there exist $u_n \in H^1_0 (\HGraph \cap B_n(x))$ such that $u_n \to u$ in $H^1$ and hence $\RQ(u_n) \to \RQ(u)$.

But this, in turn, follows immediately from the density of $H^1_c (\HGraph) := \{f \in H^1_0 (\HGraph): \supp f \text{ is bounded}\}$ in $H^1_0 (\HGraph)$. This was previously shown in \cite[Lemma~3.9]{Ho19}, but for the sake of completeness we reproduce the argument here: we let $u\in H^1_0(\HGraph)$ and define
\begin{equation}
\label{eq:psi-n-cutoff-function}
    \psi_n(y) = \frac{1}{n} (n-\min\{n, \dist_\Graph(x, y)\}),
\end{equation}    
then $\psi_n$ is continuous and piecewise smooth, with $\supp \psi_n \subset B_n(x) $. Hence, 
\begin{equation}
    (\psi_n u)' = \psi_n' u + \psi_n u'
\end{equation}
and $\psi_n u\in H_c^1(\HGraph)$. We estimate
\begin{equation}
\begin{aligned}
    \| u - \psi_n u \|_{H^1}&\le\|u'- (\psi_n u)' \|_{L^2} + \|(1+V)^{1/2}(u-\psi_n u)\|_{L^2} \\
    &\le \|\psi_n' u\|_{L^2} + \|(1- \psi_n) u'\|_{L^2} +  \|(1+V)^{1/2}(u-\psi_n u)\|_{L^2} \\
    &\le \|\psi_n'\|_\infty \|u\|_{L^2} + \left ( \int_{\HGraph \setminus B_n(x)} |u'|^2\, \mathrm dx \right )^{1/2} + \left (\int_{\HGraph \setminus B_n(x)} (1+V)|u|^2\, \mathrm dx\right )^{1/2} \\
    &\le \frac{1}{n} \| u\|_{L^2} + \left (\int_{\HGraph \setminus B_n(x)} |u'|^2\, \mathrm dx\right )^{1/2} + \left ( \int_{\HGraph\setminus B_n(x)}  (1+V) |u|^2 \, \mathrm dx\right )^{1/2},
\end{aligned}
\end{equation}
which converges to $0$ as $n \to \infty$ since $u \in H^1_0(\HGraph)$.
\end{proof}

\begin{example}
\label{ex:domain-non-continuity}
Note that the function in \eqref{eq:expanding-ball} may not be continuous if $\HGraph$ has cycles or degree one vertices which are not boundary vertices. In the simplest case where $\Graph = \HGraph$ is a cycle, say of length one, then, fixing any $x \in \Graph$, since $B_R(x)$ is a Dirichlet interval of length $2R$ for all $R < 1/2$, and $B_R(x) = \Graph$ for $R\geq 1/2$, we have $\infspec (B_R(x)) = \frac{\pi^2}{4R^2}$ for $R< 1/2$ and $\infspec (B_R(x)) = 0$ for $R\geq 1/2$.
\end{example}

\begin{lemma}
\label{lem:contracting-balls-at-infinity}
Under our standing assumptions, the mapping
\begin{displaymath}
	R \mapsto \infspec (\HGraph \setminus B_R(x))
\end{displaymath}
is a continuous and monotonically increasing function of $R \in (0,\infty)$, which satisfies
\begin{equation}
\label{eq:threshold-condition}
	\lim_{R \to \infty} \infspec (\HGraph \setminus B_R(x)) = \infess (\HGraph).
\end{equation}
In particular, $\infess (\HGraph \setminus B_R (x)) = \infess (\HGraph)$ for all $R > 0$.
\end{lemma}

\begin{proof}
Monotonicity again follows from Lemma~\ref{lem:monotonicity}(1), while continuity will be proved in Lemma~\ref{lem:domain-continuity}. Since the monotonicity of Lemma~\ref{lem:monotonicity}(1) in particular implies that for any compact $\mathcal K \subsubset \HGraph$ and any $x \in \HGraph$,
\begin{displaymath}
	\infspec (\HGraph \setminus B_R(x)) \geq \inf_{\substack{0 \neq f \in H^1_0 (\HGraph)\\ \support f \cap \mathcal K = \emptyset}} \RQ[f]
\end{displaymath}
for all sufficiently large $R>0$, the limit \eqref{eq:threshold-condition} now follows from Theorem~\ref{thm:persson} and Lemma~\ref{lem:ims}.
\end{proof}

\begin{corollary}
\label{cor:ticket-zones}
Suppose in addition to Assumption~\ref{ass:main-assumption} that $\infess (\Graph) < \infty$, let $x \in \Graph$ be fixed but arbitrary, and let $\lambda > \infess (\Graph)$. Then for any $R_1>0$ there exist $R_3>R_2 \geq R_1$ such that
\begin{displaymath}
    \infspec (A_{R_1,R_3}) \leq \lambda \qquad \text{and}
    \qquad \infspec (A_{R_2,R_3}) = \lambda.
\end{displaymath}    
\end{corollary}

\begin{proof}
We first consider $\HGraph:= \Graph \setminus B_{R_1} (x)$. By Lemma~\ref{lem:contracting-balls-at-infinity} and assumption,
\begin{displaymath}
	\infspec (\HGraph) = \infspec (\Graph \setminus B_{R_1}(x)) \leq \lim_{R \to \infty} \infspec (\Graph \setminus B_R(x)) = \infess (\Graph) < \lambda.
\end{displaymath}
We now apply Lemma~\ref{lem:expanding-balls-at-0} to $\HGraph$: since $A_{R_1,R_3} (x) = \HGraph \cap B_{R_3}(x)$, we have that $\infspec (A_{R_1,R_3}(x))$ is a monotonically decreasing function of $R_3 \in (R_1,\infty)$, which satisfies $\infspec (A_{R_1,R_3}(x)) \to \infty$ as $R_3 \to R_1$ and $\infspec (A_{R_1,R_2}(x)) \to \infspec (\HGraph) < \lambda$ as $R_3 \to \infty$. We can thus find some $R_3 > R_1$ such that $\infspec (A_{R_1,R_3} (x)) \leq \lambda$.

If there is no $R_3$ which guarantees equality, then, having fixed $R_3$ for which the inequality is strict, we finally apply Lemma~\ref{lem:contracting-balls-at-infinity} to $\HGraph = A_{R_1,R_3}$ and $B_{R_2} (0)$ for $R_2 \in [R_1,R_3)$. Now $\infspec (A_{R_2,R_3})$ is a continuous function of $R_2$, which converges to $\infspec (A_{R_1,R_3} (x)) < \lambda$ as $R_2 \to R_1$. We also claim it diverges to $\infty$ as $R_2 \to R_3$; indeed, in this case $|A_{R_2,R_3}| \to 0$, whence by Nicaise' inequality \cite[Th\'eor\`eme~3.1]{Ni87} applied to the compact graph $A_{R_2,R_3}$
\begin{displaymath}
    \infspec (A_{R_2,R_3}) \geq \frac{\pi^2 c^2}{4|A_{R_2,R_3}|} \to \infty,
\end{displaymath}
where we have estimated $V$ from below by $0$ and used that $A_{R_2,R_3}$ always has at least one boundary point and thus Dirichlet vertex. The existence of some $R_2 \in (R_1,R_3)$ such that $\infspec (A_{R_2,R_3}) = \lambda$ follows.
\end{proof}

\section{Proof of the main theorems}
\label{sec:proofs}

\begin{proof}[Proof of Theorem~\ref{thm:less_than_ess}]
We may obviously suppose that $\infess (\Graph) < \infty$. Fix $k \in \N$ and $\lambda > \infess (\Graph)$; we will find a $k$-partition $\Partition = (\Graph_1,\ldots,\Graph_k)$ such that
\begin{displaymath}
	\denergy{k} (\Partition) \leq \lambda;
\end{displaymath}
the conclusion of the theorem then follows immediately.

Fix any point $0 \in \Graph$, without loss of generality a vertex. Now since $\lambda> \infess (\Graph) \geq \infspec (\Graph)$, by the monotonicity and limit properties from Lemma~\ref{lem:expanding-balls-at-0} applied to $\HGraph = \Graph$ there exists $R_1 > 0$ such that
\begin{displaymath}
	\infspec (B_{R_1} (0)) \leq \lambda;
\end{displaymath}
we then set $\Graph_1 := B_{R_1} (0)$ and observe that $\Graph_1$ is necessarily connected. We next apply Corollary~\ref{cor:ticket-zones} successively to find radii $R_1 < R_2 < \ldots < R_k$ such that
\begin{displaymath}
	\infspec (A_{R_{i-1},R_i} (0)) \leq \lambda
\end{displaymath}
for all $i=2,\ldots,k$. Noting that $A_{R_{i-1},R_i} (0)$ is compact but may not be connected, there exists a connected component, call it $\Graph_i$, on which the ground state of $A_{R_{i-1},R_i} (0)$ is supported, that is, such that $\infspec (\Graph_i) = \infspec (A_{R_{i-1},R_i} (0)) \leq \lambda$. We choose $(\Graph_1,\ldots,\Graph_k)$, thus defined, to be our partition.
\end{proof}

\begin{proof}[Proof of Theorem~\ref{thm:existenceprinciple}]
Let $\Partition_n = (\Graph_{1,n}, \ldots, \Graph_{k,n})$ be a sequence of $k$-partitions whose energy approaches the infimum $\doptenergy{k}$; we may assume the existence of some $\lambda < \infess (\Graph)$ such that $\denergy{k} (\Partition_n) \leq \lambda$ for all $n \geq 1$. We fix an arbitrary point $0 \in \Graph$. The proof of existence is divided into three steps; the finite ball property contained in Assumption~\ref{ass:main-assumption} will enter explicitly in Steps 2 and 3.

\textit{Step 1: for any sufficiently large $r>0$, $\Graph_{i,n}$ has nontrivial intersection with $B_r(0)$ for all $i$ and $n$.} Indeed, by Lemma~\ref{lem:contracting-balls-at-infinity}, for any $r>0$ sufficiently large, $\infspec (\Graph \setminus B_r(0)) > \lambda$. Fix any such $r>0$. If, for some $i = 1,\ldots, k$ and $n \geq 1$, we should have $\Graph_{i,n} \subset \Graph \setminus B_r(0)$, then by domain monotonicity, Lemma~\ref{lem:monotonicity}(1),
\begin{displaymath}
    \lambda < \infspec (\Graph \setminus B_r(0)) \leq \infspec (\Graph_{i,n}) \leq \denergy{k} (\Partition_n) \leq \lambda,
\end{displaymath}
a contradiction.

\textit{Step 2: up to a subsequence there exists a limit partition $\Partition$.} More precisely, we will show that there exists a $k$-partition $\Partition = (\Graph_1, \ldots, \Graph_k)$ such that, up to a subsequence, on every precompact set $\mathcal{K}$ of $\Graph$, $\Graph_{i,k} \cap \mathcal{K} \to \Graph_i \cap \mathcal{K}$ in the sense of \cite[Section~3]{KKLM21}.

For each $n \geq 1$, $B_n(0)$ is a compact, connected subset of $\Graph$, which as a subgraph contains a finite number of edges by the finite ball property; hence there are only finitely many \emph{configuration classes} inside $B_n(0)$ in the sense of \cite[Definition~3.3]{KKLM21}. Thus, as in the proof of \cite[Theorem~3.13]{KKLM21}, up to a subsequence the restricted partitions $\Partition_{n,n} := (\Graph_{1,n} \cap B_n(0), \ldots, \Graph_{k,n} \cap B_n(0))$ admit a limit $\Partition_n^\ast = (\Graph_{1,n}^\ast,\ldots,\Graph_{k,n}^\ast)$; by Step 1, for $n \geq 1$ sufficiently large $\Partition_n^\ast$ is a $k$-partition. (Note that this argument is purely topological and does \emph{not} involve the partition energies or the underlying operators; thus the proof of  \cite[Theorem~3.13]{KKLM21} may be repeated verbatim in our case.)

By a diagonal argument, we can ensure that whenever $m>n$, for all $i=1,\ldots,k$, $\Graph_{i,m}^\ast \cap B_n(0)$ and $\Graph_{i,n}^\ast$ coincide. In particular, the limit $\Graph_i$ is well-defined as
\begin{displaymath}
    \Graph_i = \bigcup_{m\geq 1} \lim_{n \to \infty} \Graph_{i,n} \cap B_m (0)
\end{displaymath}
(where for each $m \in \N$ the limit is again understood in the sense of \cite[Section~3]{KKLM21}, inside the fixed compact graph $B_m(0)$). We then set $\Partition := (\Graph_1, \ldots, \Graph_k)$; by construction, this partition is the limit in the sense claimed above.

\textit{Step 3: lower semicontinuity with respect to partition convergence.} We show that up to a further subsequence 
\begin{equation}
\label{eq:step-3}
\infspec (\Graph_i) \le \limsup_{n\to \infty} \infspec (\Graph_{i,n})  
\end{equation}
for all $i=1,\ldots,k$, from which the conclusion of the theorem will follow immediately since 
\begin{equation}
    \doptenergy{k}(\Graph)\le \denergy{k}(\Partition)= \max_{i=1,\ldots, k} \infspec (\Graph_i) \le \lim_{n\to \infty} \max_{i=1,\ldots, k} \infspec(\Graph_i) = \doptenergy{k}(\Graph).  
\end{equation}

Fix $i=1,\ldots,k$. We first observe that, for each fixed $n \geq 1$, by Lemma~\ref{lem:expanding-balls-at-0} applied to $\HGraph = \Graph_{i,n}$, we have $\infspec (\Graph_{i,n} \cap B_r(0)) \to \infspec (\Graph_{i,n})$ as $r \to \infty$; a corresponding statement is true for $\Graph_i$.

At the same time, for each fixed $r>0$ large enough the nature of the convergence $\Graph_{i,n} \cap B_r(0) \to \Graph_i \cap B_r(0)$ implies that each boundary vertex in $\Graph_{i}\cap B_r(0)$ admits a sequence of boundary vertices of $\Graph_{i,n} \cap B_r(0)$ that converge towards it in the metric space $\Graph$. Then by Corollary~\ref{cor:hgraph-mosco-convergence} we infer 
\begin{displaymath}
\infspec (\Graph_i \cap B_r(0)) \le \limsup_{n\to \infty} \infspec (\Graph_{i,n} \cap B_r(0)) 
\end{displaymath}
as $n \to \infty$ for each $r>0$ large enough.

Finally, we may assume without loss of generality (after passing to a further subsequence if necessary) that $\infspec (\Graph_{i,n})$ forms a Cauchy sequence in $n\geq 1$, since its values are bounded in the interval $[\doptenergy{k}(\Graph), \lambda]$. In particular, up to a further subsequence, $\infspec (\Graph_{i,n} \cap B_n (0))$ is also Cauchy and has the same limit. Using monotonicity in $r$, which implies (possibly up to a further subsequence) that
\begin{displaymath}
    \lim_{r\to\infty} \limsup_{n\to\infty} \infspec (\Graph_{i,n} \cap B_r(0))
    = \lim_{n\to\infty} \infspec (\Graph_{i,n} \cap B_n(0))
    = \lim_{n\to\infty} \infspec (\Graph_{i,n}),
\end{displaymath}
we thus finally obtain, for this subsequence,
\begin{displaymath}
    \infspec (\Graph_i) = \lim_{n\to\infty} \infspec (\Graph_i \cap B_n (0))
    \leq \lim_{n\to\infty} \infspec (\Graph_{i,n} \cap B_n(0))
    = \lim_{n\to\infty} \infspec (\Graph_{i,n})
\end{displaymath}
for all $i=1,\ldots,k$, as claimed, from which \eqref{eq:step-3} follows. This completes the proof that $\doptenergy{k}(\Graph) = \denergy{k} (\Partition)$, and hence that $\Partition$ is a $k$-partition such that $\denergy{k} (\Partition) = \lim_{n \to \infty} \denergy{k} (\Partition_n) = \doptenergy{k} (\Graph)$.

Now, by Lemma~\ref{lem:monotonicity}(3) and by assumption, for each $i$ we have $\infess (\Graph_i) \geq \infess (\Graph) > \infspec (\Graph_i)$. It now follows immediately from standard theory for Schrödinger operators that $\infspec (\Graph_i)$ is an isolated eigenvalue with a positive eigenfunction (see, e.g., \cite{Ku19}).
\end{proof}

\section{Examples}
\label{sec:examples}

From our main results we know that only two scenarios are possible. Either $\doptenergy{k}(\Graph)<\Sigma(\Graph)$ and spectral minimal $k$-partitions exist, or else $\doptenergy{k}(\Graph)=\Sigma(\Graph)$. Here we will discuss three examples which showcase what can happen when there is equality.

The first is a simple one which shows that spectral minimal $k$-partitions may exist for some $k$, but not for others. The second and third are based on so-called \emph{rooted equilateral tree graphs}, which have been considered in many works, including \cite{DST20, Ho21, KoMuNi19,KoNi19,SoSo02,So04}, among others. The second, which admits a minimal $k$-partition for any $k \geq 1$, also illustrates that on graphs $\doptenergy{k} (\Graph) = \infess (\Graph) > 0$ is possible for all $k \geq 1$ even if $V \equiv 0$. The third shows that if a minimal $k$-partition exists, then the clusters may or may not admit ground states.

\begin{example}
\label{ex:main}
Take $\Graph = \R$ and $V \equiv 0$; then for all $k \geq 1$ we clearly have
\begin{displaymath}
	\doptenergy{k}(\R) = \infess (\R) = 0.
\end{displaymath}
When $k=2$, for any $x \in \R$ the $2$-partition $\Partition_2^x := ( (-\infty,x], [x,\infty) )$ is minimal, although in this case the infimum of the spectrum is not an eigenvalue. However, for $k\geq 3$ there are no partitions whose spectral energy is $0$ since necessarily at least one partition element needs to be bounded and hence has a strictly positive associated ground state energy.

More generally, if $\mathcal{S}_m$ denotes the $m$-star graph consisting of $m$ semi-axes, or \emph{rays}, joined at a common vertex, then we still have
\begin{equation}
\label{eq:rays}
	\doptenergy{k}(\mathcal{S}_m) = \infess (\mathcal{S}_m) = 0
\end{equation}
for all $k \in \N$; there exists a minimizing $k$-partition if and only if $k \leq m$. Any graph satisfying Assumption~\ref{ass:main-assumption} which has at least $m \geq 1$ such rays satisfies \eqref{eq:rays} for all $k \in \N$, and admits a minimizing $k$-partition for all $k \leq m$ (this includes the graphs sometimes called \emph{starlike}, as in \cite{CaFiNo17}).
\end{example}

\begin{example}
\label{ex:v0ls}
Fix $k \geq 2$. We start by considering a homogeneous rooted tree $\Tree$, i.e. a regular tree with constant branching number $b$ and constant edge length $1$ (see Figure~\ref{fig:regular_tree_graph}), with Neumann condition at the root, and equipped with the Laplacian, i.e. $V \equiv 0$.
\begin{figure}[ht]
    \centering
    \begin{minipage}{.375\textwidth}\includegraphics[clip, trim=5cm 0cm 5cm 0cm, scale=0.8]{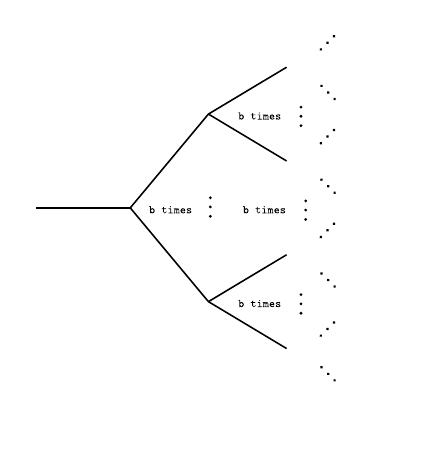}\end{minipage}
    \begin{minipage}{.375\textwidth}\includegraphics[clip, trim=5cm 0cm 5cm 0cm,scale=0.8]{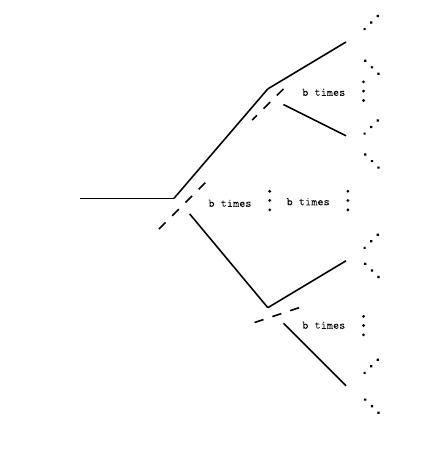}\end{minipage}\vspace{-12pt}
    \caption{Left: a rooted tree $\Tree$ with branching number $b$. Right: a sketch of how to find an infinite spectral minimal partition by ``cutting branches''. }
    \label{fig:regular_tree_graph}
\end{figure}

Then it is known (see \cite[Theorem~3.3]{SoSo02}) that the spectrum of $\Tree$ is purely essential with $\infspec(\Tree) = \infess (\Tree) = \theta^2 > 0$, where
\begin{displaymath}
    \theta = \arccos \left( \frac{2}{b^{1/2} + b^{-1/2}} \right),
\end{displaymath}
and there are no eigenfunctions corresponding to this value.

It follows from Lemma~\ref{lem:monotonicity} that $\infspec (\HGraph) \geq \infspec (\Tree) = \infess (\Tree) > 0$ for any subgraph $\HGraph \subset \Tree$, and thus $\doptenergy{k} (\Tree) = \infess (\Tree)$ for all $k \geq 1$. However, for any fixed $k$, we can find $k$ copies $\Tree_1, \ldots, \Tree_k$ of $\Tree$ as imbedded subgraphs of $\Tree$ (cf.\ Figure~\ref{fig:regular_tree_graph}) with a Dirichlet rather than a Neumann condition at their root vertex (which is their only boundary vertex as subgraphs). This does not affect the essential spectrum, and thus $\infess (\Tree) = \infess (\Tree_i) = \infspec (\Tree_i) \geq \infspec (\Tree_i) = \infess (\Tree)$ for all $i$. In particular, together they form a $k$-minimal partition.
\end{example}

\begin{example}
\label{ex:startree}
We take $\Tree$ as in the previous example and continue to assume $V \equiv 0$; however, we impose a Dirichlet condition at the root, which clearly does not affect either the value or the nature of $\theta$. Given $k \geq 2$, we will modify $\Tree$ to create a graph $\Graph_k$ by attaching $k$ intervals to $\Tree$ at the root (see Figure~\ref{fig:regular_tree_graph_intervals}).
\begin{figure}[ht]
    \centering
    \includegraphics[scale=0.8]{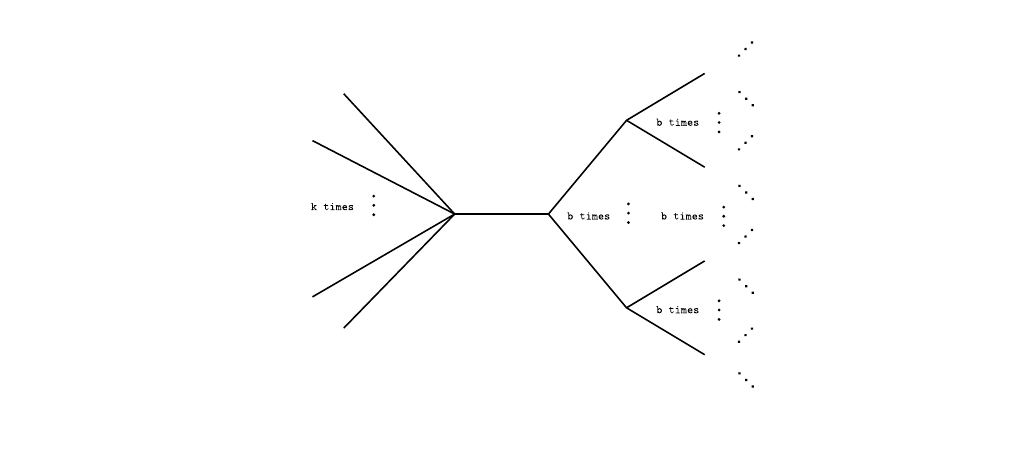}\vspace{-24pt}
    \caption{The graph $\Graph_k$ formed by gluing a rooted tree $\Tree$ with intervals with equal ground state energy at the root vertex.}
    \label{fig:regular_tree_graph_intervals}
\end{figure}

More precisely, we consider intervals $\mathcal I_j, j = 1, \dots k$, each of length $\ell = \pi / (2 \theta)$, and each equipped with the Laplacian with a Dirichlet at one endpoint and a Neumann condition at the other. We let $\widetilde\Graph_k := \Tree \cup \bigcup_j \mathcal I_j$ be the disjoint union of the tree and these intervals.

Now the spectrum of $\mathcal I_j$ is purely discrete; indeed, $\sigma(\mathcal I_j) = \{ n^2 \theta^2 : n \in \N \}$. So the Laplacian on $\widetilde\Graph_k$ has spectrum $\sigma(\Tree \cup \bigcup_j \mathcal I_j) = \ess(\Tree) \cup \bigcup_{j=1}^{k} \sigma(\mathcal I_j)$ where the eigenvalues are counted with multiplicities. We now glue together all $k+1$ Dirichlet points at the root of the tree, which we consider to be equipped with standard conditions in the resulting graph $\Graph_k$. The common gluing vertex will be called the central vertex. The operation of gluing is a rank one perturbation of the standard Laplacian which pushes down the eigenvalues according to a standard interlacing inequality (see \cite[Theorems~3.1.10 and~3.1.11]{BeKu13}), thus
\begin{equation}
    \lambda_1(\Graph_k) \leq \theta^2 = \lambda_2(\Graph_k) = \dots = \lambda_k(\Graph_k).
\end{equation}

We claim that any partition $\Partition_k$ of $\Graph_k$ whose clusters are any $k$ of the $k+1$ connected components of $\widetilde\Graph_k$ is a minimal $k$-partition of $\Graph_k$. Combining this with the cutting principle in Example~\ref{ex:v0ls}, for any $j=0,\ldots,k$ we can thus find a minimal $k$-partition $\Partition$ such that exactly $j$ clusters of $\Partition$ admit a ground state.

We now prove the claim. By construction, each interval graph $\mathcal I_j, j=1,\dots,k$ with Kirchhoff-Neumann conditions at the endpoints shares the same ground state energy $\theta^2$, which is also the infimum of the spectrum of $\Tree$. We need to show that $\Partition_k$ is minimal, i.e. $\denergy{k} (\Partition_k) = \doptenergy{k} (\Graph_k)$.
Suppose $\widetilde \Partition_k$ is a $k$-partition with a set of cut points different than the central vertex, then at least one of the cut points must belong to one of the connected components of $\widetilde\Graph_k$. Then either one of the partition elements would be an interval smaller than $\pi/(2\theta)$ or one partition element would be contained in $\Tree$. Hence, $\denergy{k}(\Partition) \ge \denergy{k}(\Partition_k)= \theta$ and $\Partition_k$ is a spectral minimal partition.
\end{example}

\appendix
\section{Continuity of the infimum of the spectrum with respect to graph perturbation}
\label{sec:appendix-continuity}

Here we will show how the infimum of the spectrum of a Schr\"odinger operator on a subgraph $\HGraph$ of a given graph $\Graph$, depends continuously, or semicontinuously, on $\HGraph$, in particular to complete the proofs of Lemma~\ref{lem:expanding-balls-at-0} and Lemma~\ref{lem:contracting-balls-at-infinity}: more precisely, for a given metric graph $\Graph$ satisfying the finite ball condition of Assumption~\ref{ass:main-assumption}, a fixed subgraph $\HGraph$ of $\Graph$, and any $x \in \Graph$, the map $R \mapsto \infspec (\HGraph \cap B_R(x))$ is a lower semicontinuous function of $R \in (0,\infty)$, while $R \mapsto \infspec (\HGraph \setminus B_R(x))$ is continuous. (We recall that $B_R (x)$ is the \emph{closed} ball of Definition~\ref{def:balls}, and that as subgraphs $\HGraph \cap B_R(x)$ and $\HGraph \setminus B_R(x)$ are assumed to contain no isolated points.) To this end we fix a graph $\Graph$ and a potential $V$ satisfying Assumption~\ref{ass:main-assumption}.

We start by formalizing a natural notion of convergence of subgraphs, in terms of the behavior of their boundaries. We first need to distinguish subsets of a sequence of non-connected subgraphs which might disappear in the limit (as happened to the complement of $B_R(x)$ in Example~\ref{ex:domain-non-continuity}), and which are thus irrelevant for spectral purposes as their ground state energy diverges. This leads to the following notion. We recall that the boundary of a subgraph is always its topological boundary as a subset of $\Graph$.

\begin{definition}
Let $\HGraph_n$ be any sequence of subgraphs of a fixed graph $\Graph$. We will say that a sequence of boundary points $v_n \in \partial\HGraph_n$ is \emph{vanishing} if, for any $r>0$,
\begin{displaymath}
    |\HGraph_n \cap B_r (v_n)| \to 0,
\end{displaymath}
and \emph{non-vanishing} otherwise.
\end{definition}

This will only be of interest when $v_n$ forms a Cauchy sequence in $\Graph$. 

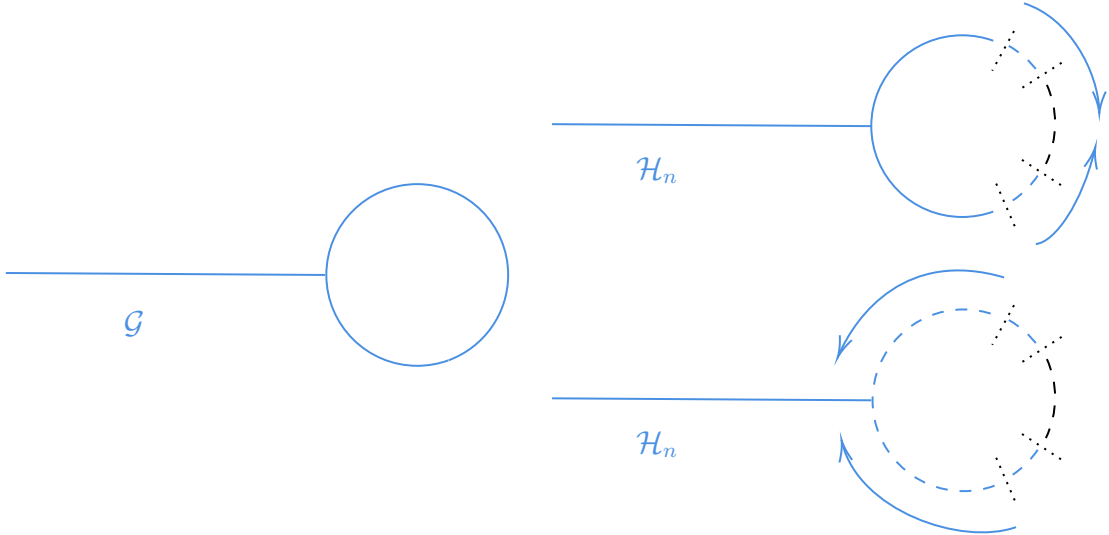
\begin{figure}[ht]
    \centering

\begin{minipage}{.4\textwidth}
\tikzset{every picture/.style={line width=0.75pt}} 

\begin{tikzpicture}[x=0.75pt,y=0.75pt,yscale=-1,xscale=1]

\draw [color={rgb, 255:red, 74; green, 144; blue, 226 }  ,draw opacity=1 ]   (100,120) -- (261,121) ;
\draw  [draw opacity=0] (323.19,164.1) .. controls (318.29,165.88) and (313.01,166.85) .. (307.5,166.85) .. controls (282.18,166.85) and (261.65,146.32) .. (261.65,121) .. controls (261.65,95.68) and (282.18,75.15) .. (307.5,75.15) .. controls (332.82,75.15) and (353.35,95.68) .. (353.35,121) .. controls (353.35,140.74) and (340.88,157.57) .. (323.38,164.03) -- (307.5,121) -- cycle ; \draw  [color={rgb, 255:red, 74; green, 144; blue, 226 }  ,draw opacity=1 ] (323.19,164.1) .. controls (318.29,165.88) and (313.01,166.85) .. (307.5,166.85) .. controls (282.18,166.85) and (261.65,146.32) .. (261.65,121) .. controls (261.65,95.68) and (282.18,75.15) .. (307.5,75.15) .. controls (332.82,75.15) and (353.35,95.68) .. (353.35,121) .. controls (353.35,140.74) and (340.88,157.57) .. (323.38,164.03) ;  

\draw (158,138) node [anchor=north west][inner sep=0.75pt]  [color={rgb, 255:red, 74; green, 144; blue, 226 }  ,opacity=1 ] [align=left] {$\displaystyle \mathcal{G}$};

\end{tikzpicture}    
\end{minipage}
\ \quad \ 
\begin{minipage}{.4\textwidth}
\tikzset{every picture/.style={line width=0.75pt}} 

\begin{tikzpicture}[x=0.75pt,y=0.75pt,yscale=-1,xscale=1]

\draw [color={rgb, 255:red, 74; green, 144; blue, 226 }  ,draw opacity=1 ]   (100,120) -- (261,121) ;
\draw  [draw opacity=0] (322.54,164.1) .. controls (317.65,165.88) and (312.36,166.85) .. (306.85,166.85) .. controls (281.53,166.85) and (261,146.32) .. (261,121) .. controls (261,95.68) and (281.53,75.15) .. (306.85,75.15) .. controls (312.36,75.15) and (317.65,76.12) .. (322.54,77.9) -- (306.85,121) -- cycle ; \draw  [color={rgb, 255:red, 74; green, 144; blue, 226 }  ,draw opacity=1 ] (322.54,164.1) .. controls (317.65,165.88) and (312.36,166.85) .. (306.85,166.85) .. controls (281.53,166.85) and (261,146.32) .. (261,121) .. controls (261,95.68) and (281.53,75.15) .. (306.85,75.15) .. controls (312.36,75.15) and (317.65,76.12) .. (322.54,77.9) ;  
\draw  [draw opacity=0][dash pattern={on 4.5pt off 4.5pt}] (349.56,100.19) .. controls (352.22,106.11) and (353.7,112.67) .. (353.7,119.58) .. controls (353.7,126.92) and (352.03,133.87) .. (349.05,140.07) -- (306.34,119.58) -- cycle ; \draw  [dash pattern={on 4.5pt off 4.5pt}] (349.56,100.19) .. controls (352.22,106.11) and (353.7,112.67) .. (353.7,119.58) .. controls (353.7,126.92) and (352.03,133.87) .. (349.05,140.07) ;  
\draw  [dash pattern={on 0.84pt off 2.51pt}]  (324,150) -- (334,173) ;
\draw  [dash pattern={on 0.84pt off 2.51pt}]  (333,73) -- (322,93) ;
\draw  [dash pattern={on 0.84pt off 2.51pt}]  (357,89) -- (335,103) ;
\draw  [dash pattern={on 0.84pt off 2.51pt}]  (337,138) -- (357,151) ;
\draw  [draw opacity=0][dash pattern={on 4.5pt off 4.5pt}] (344.95,145.96) .. controls (341.62,150.94) and (337.33,155.22) .. (332.34,158.53) -- (307.5,121) -- cycle ; \draw  [color={rgb, 255:red, 74; green, 144; blue, 226 }  ,draw opacity=1 ][dash pattern={on 4.5pt off 4.5pt}] (344.95,145.96) .. controls (341.62,150.94) and (337.33,155.22) .. (332.34,158.53) ;  
\draw  [draw opacity=0][dash pattern={on 4.5pt off 4.5pt}] (330.35,80.49) .. controls (336.36,83.89) and (341.55,88.59) .. (345.51,94.21) -- (307.5,121) -- cycle ; \draw  [color={rgb, 255:red, 74; green, 144; blue, 226 }  ,draw opacity=1 ][dash pattern={on 4.5pt off 4.5pt}] (330.35,80.49) .. controls (336.36,83.89) and (341.55,88.59) .. (345.51,94.21) ;  
\draw [color={rgb, 255:red, 74; green, 144; blue, 226 }  ,draw opacity=1 ]   (344,180.5) .. controls (353.8,179.52) and (367.44,159.33) .. (373.63,133.11) ;
\draw [shift={(374,131.5)}, rotate = 102.53] [color={rgb, 255:red, 74; green, 144; blue, 226 }  ,draw opacity=1 ][line width=0.75]    (10.93,-3.29) .. controls (6.95,-1.4) and (3.31,-0.3) .. (0,0) .. controls (3.31,0.3) and (6.95,1.4) .. (10.93,3.29)   ;
\draw [color={rgb, 255:red, 74; green, 144; blue, 226 }  ,draw opacity=1 ]   (338,59) .. controls (361.28,66.28) and (375.15,93.78) .. (375.96,112.76) ;
\draw [shift={(376,114.5)}, rotate = 270] [color={rgb, 255:red, 74; green, 144; blue, 226 }  ,draw opacity=1 ][line width=0.75]    (10.93,-3.29) .. controls (6.95,-1.4) and (3.31,-0.3) .. (0,0) .. controls (3.31,0.3) and (6.95,1.4) .. (10.93,3.29)   ;

\draw (141,135) node [anchor=north west][inner sep=0.75pt]  [color={rgb, 255:red, 74; green, 144; blue, 226 }  ,opacity=1 ] [align=left] {$\displaystyle \mathcal{H}_{n}$};

\end{tikzpicture}

\tikzset{every picture/.style={line width=0.75pt}} 

\begin{tikzpicture}[x=0.75pt,y=0.75pt,yscale=-1,xscale=1]

\draw [color={rgb, 255:red, 74; green, 144; blue, 226 }  ,draw opacity=1 ]   (100,120) -- (261,121) ;
\draw  [draw opacity=0][dash pattern={on 4.5pt off 4.5pt}] (323.19,164.1) .. controls (318.29,165.88) and (313.01,166.85) .. (307.5,166.85) .. controls (282.18,166.85) and (261.65,146.32) .. (261.65,121) .. controls (261.65,95.68) and (282.18,75.15) .. (307.5,75.15) .. controls (313.01,75.15) and (318.29,76.12) .. (323.19,77.9) -- (307.5,121) -- cycle ; \draw  [color={rgb, 255:red, 74; green, 144; blue, 226 }  ,draw opacity=1 ][dash pattern={on 4.5pt off 4.5pt}] (323.19,164.1) .. controls (318.29,165.88) and (313.01,166.85) .. (307.5,166.85) .. controls (282.18,166.85) and (261.65,146.32) .. (261.65,121) .. controls (261.65,95.68) and (282.18,75.15) .. (307.5,75.15) .. controls (313.01,75.15) and (318.29,76.12) .. (323.19,77.9) ;  
\draw  [draw opacity=0][dash pattern={on 4.5pt off 4.5pt}] (349.56,100.19) .. controls (352.22,106.11) and (353.7,112.67) .. (353.7,119.58) .. controls (353.7,126.92) and (352.03,133.87) .. (349.05,140.07) -- (306.34,119.58) -- cycle ; \draw  [dash pattern={on 4.5pt off 4.5pt}] (349.56,100.19) .. controls (352.22,106.11) and (353.7,112.67) .. (353.7,119.58) .. controls (353.7,126.92) and (352.03,133.87) .. (349.05,140.07) ;  
\draw  [dash pattern={on 0.84pt off 2.51pt}]  (324,150) -- (334,173) ;
\draw  [dash pattern={on 0.84pt off 2.51pt}]  (333,73) -- (322,93) ;
\draw  [dash pattern={on 0.84pt off 2.51pt}]  (357,89) -- (335,103) ;
\draw  [dash pattern={on 0.84pt off 2.51pt}]  (337,138) -- (357,151) ;
\draw  [draw opacity=0][dash pattern={on 4.5pt off 4.5pt}] (344.95,145.96) .. controls (341.62,150.94) and (337.33,155.22) .. (332.34,158.53) -- (307.5,121) -- cycle ; \draw  [color={rgb, 255:red, 74; green, 144; blue, 226 }  ,draw opacity=1 ][dash pattern={on 4.5pt off 4.5pt}] (344.95,145.96) .. controls (341.62,150.94) and (337.33,155.22) .. (332.34,158.53) ;  
\draw  [draw opacity=0][dash pattern={on 4.5pt off 4.5pt}] (330.35,80.49) .. controls (336.36,83.89) and (341.55,88.59) .. (345.51,94.21) -- (307.5,121) -- cycle ; \draw  [color={rgb, 255:red, 74; green, 144; blue, 226 }  ,draw opacity=1 ][dash pattern={on 4.5pt off 4.5pt}] (330.35,80.49) .. controls (336.36,83.89) and (341.55,88.59) .. (345.51,94.21) ;  
\draw [color={rgb, 255:red, 74; green, 144; blue, 226 }  ,draw opacity=1 ]   (334,185) .. controls (306.42,194.85) and (254.58,176.56) .. (246.35,142.56) ;
\draw [shift={(246,141)}, rotate = 78.69] [color={rgb, 255:red, 74; green, 144; blue, 226 }  ,draw opacity=1 ][line width=0.75]    (10.93,-3.29) .. controls (6.95,-1.4) and (3.31,-0.3) .. (0,0) .. controls (3.31,0.3) and (6.95,1.4) .. (10.93,3.29)   ;
\draw [color={rgb, 255:red, 74; green, 144; blue, 226 }  ,draw opacity=1 ]   (328,59) .. controls (278.53,44.45) and (253.52,77.89) .. (244.77,97.7) ;
\draw [shift={(244,99.5)}, rotate = 292.31] [color={rgb, 255:red, 74; green, 144; blue, 226 }  ,draw opacity=1 ][line width=0.75]    (10.93,-3.29) .. controls (6.95,-1.4) and (3.31,-0.3) .. (0,0) .. controls (3.31,0.3) and (6.95,1.4) .. (10.93,3.29)   ;

\draw (141,135) node [anchor=north west][inner sep=0.75pt]  [color={rgb, 255:red, 74; green, 144; blue, 226 }  ,opacity=1 ] [align=left] {$\displaystyle \mathcal{H}_{n}$};

\end{tikzpicture}
\end{minipage}
    \caption{Tadpole (left) with sequences of subgraphs $\HGraph_n$ (top and bottom right). The sequence of subgraphs top right does not converge topologically towards the original graph, since in the limit graph $\Graph$ the boundary points disappear and condition (2) is not satisfied. The sequence of subgraphs bottom right does converge topologically towards the interval (solid blue line). }
    \label{fig:vanishing_boundary}
\end{figure}

We can now give a topological notion of convergence.

\begin{definition}
\label{def:graph-convergence}
Let $\HGraph_n$, $\HGraph$ be subgraphs of a fixed graph $\Graph$. Then we will say that $\HGraph_n \to \HGraph$ topologically if the following two conditions are satisfied:
\begin{enumerate}
    \item For all $v \in \partial\HGraph$ there exist $v_n \in \partial \HGraph_n$ such that $v_n \to v$ in the metric space $\Graph$;
    \item Whenever $v_n \in \partial\HGraph_n$ is a bounded sequence of non-vanishing boundary points of $\HGraph_n$ in $\Graph$, there exists some $v \in \partial\HGraph$ such that, up to a subsequence, $v_n \to v$.
\end{enumerate}
\end{definition}

\begin{example}
Consider a tadpole (a loop attached to a pendant interval) and suppose $\HGraph_n $ is a sequence of three-stars whose boundary points approach each other as $n\to \infty$ (see Figure~\ref{fig:vanishing_boundary}), then the sequence of boundary points is non-vanishing, but the sequence of subgraphs does not converge towards the original graph in the topological sense. While property (1) holds, property (2) is not satisfied.
\end{example}

To this we now add a suitable notion of \emph{form convergence}. This will be a version of \emph{Mosco convergence} for the spaces $H^1_0 (\HGraph_n)$ and $H^1_0 (\HGraph)$ (as defined in \eqref{eq:h10}, cf.\ \eqref{eq:h1}). As customary, we will regard all these spaces as being subspaces of $H^1(\Graph)$ via extension by zero of the functions. We will show below that many nested subgraphs $\HGraph_n$ and $\HGraph$ satisfy Mosco convergence in this sense (see Lemma~\ref{lem:hgraph-mosco-convergence}); this will, in particular, allow us to treat balls and exteriors of balls.

\begin{definition}
\label{def:mosco}
Given subgraphs $\HGraph_n$ and $\HGraph$ of $\Graph$, we say that $\HGraph_n \to \HGraph$ \emph{in the sense of Mosco} if the following two conditions are satisfied:
\begin{enumerate}
    \item For all $u \in H^1_0 (\HGraph)$ there exist $u_n \in H^1_0 (\HGraph)$ such that $u_n \to u$ in $H^1(\Graph)$;
    \item For all $u_n \in H^1_0 (\HGraph_n)$ such that $u_n \rightharpoonup u$ for some $u \in H^1(\HGraph)$, we have $u \in H^1_0 (\HGraph)$. 
\end{enumerate}
\end{definition}

Our first result states that convergence in the sense of Mosco, as just defined, implies convergence of the infimum of the spectrum.

\begin{lemma}
\label{lem:hgraph-ev-convergence}
Let $\HGraph$ and $\HGraph_n$ be subgraphs of a given graph $\Graph$ which agree outside a compact set, i.e., we assume there exists some bounded set $\mathcal{K} \subset \Graph$ such that $\HGraph \setminus \mathcal{K} = \HGraph_n \setminus \mathcal{K}$. Suppose that $\HGraph_n$ and $\HGraph$ satisfy condition (1) (respectively, (2)) of Definition~\ref{def:mosco}. Then
\begin{displaymath}
	\infspec (\HGraph) \geq \limsup_{n\to\infty} \infspec (\HGraph_n)
\end{displaymath}
(respectively, $\infspec (\HGraph) \leq \liminf_{n\to\infty} \infspec (\HGraph_n)$). In particular, if $\HGraph_n \to \HGraph$ in the sense of Mosco, then
\begin{displaymath}
    \infspec (\HGraph_n) \to \infspec (\HGraph).
\end{displaymath}
\end{lemma}

The condition that $\HGraph$ and $\HGraph_n$ coincide outside a compact set will only be used in case (2).

\begin{proof}
Assuming (1), fix $\varepsilon > 0$ and $u \in H^1_0(\HGraph)$ such that $\RQ (u) \leq \infspec (\HGraph) + \varepsilon$. Then we can find $u_n \in H^1_0(\HGraph_n)$ such that we have $u_n \to u$ in $H^1(\Graph)$; in particular, $\RQ (u_n) \to \RQ (u)$. It follows immediately that $\limsup_{n\to\infty} \infspec(\HGraph_n) \leq \infspec (\HGraph) +\varepsilon$.

In case (2), suppose $\limsup_{n\to \infty}\infspec(\HGraph_n)<\infty $ since there is nothing to show otherwise. Fix $\varepsilon>0$ and for each $n$ let $u_n\in H^1_0 (\HGraph_n)$ such that $\RQ (u_n) \le  \infspec(\HGraph_n)+\varepsilon$ for all $n$. Hence, by our assumption on $\infspec(\HGraph_n)$, $(u_n)$ is a bounded sequence in $H^1(\Graph)$; in particular, up to a subsequence it admits a weak limit $u\in H^1(\Graph)$; by (2), $u\in H^1_0(\HGraph)$.

Fix an arbitrary root vertex $0 \in \HGraph$ and let
\begin{equation}\label{eq:seqparuni}
    \begin{aligned}
    \psi_n &= \frac{\frac{1}{n}\max\{ \dist_\Graph(\Graph\setminus B_{2n}(0), x), n\}}{\big[(\frac{1}{n}\max\{ \dist_\Graph(\Graph\setminus B_{2n}(0), x), n\})^2+ (1-\frac{1}{n}\max\{ \dist_\Graph(\Graph\setminus B_{2n}(0), x), n\})^2\big]^{1/2}}, \\
    \tilde \psi_n &= \frac{1-\frac{1}{n}\max\{ \dist_\Graph(\Graph\setminus B_{2n}(0), x), n\}}{\big[(\max\{ \dist_\Graph(\frac{1}{n}\Graph\setminus B_{2n}(0), x), n\})^2+ (1-\frac{1}{n}\max\{ \dist_\Graph(\Graph\setminus B_{2n}(0), x), n\})^2\big]^{1/2}}.
    \end{aligned}
\end{equation}
Then $\psi_n, \tilde \psi_n$ are continuous, piecewise smooth cut-off functions $\psi_n,\tilde\psi_n$ such that $0 \leq \psi_n,\tilde\psi_n \leq 1$, $\psi_n^2 + \tilde\psi_n^2 \equiv 1$, $\supp \psi_n \subset \HGraph \setminus B_n(0)$ and $\supp \tilde\psi_n \subset \HGraph \cap B_{2n}(0)$. Passing to a subsequence in $u_n$ still denoted by $u_n$ we have
\begin{equation}
\label{eq:help-james-please}
\int_{\HGraph} |\psi_n u_n|^2 \, \dx \to \int_\HGraph |u|^2\, \dx .
\end{equation}
We also have
\begin{equation}
    \| \psi_n' \|_\infty \le \frac{C}{n}, \quad \|\widetilde \psi_n' \|_\infty\le \frac{C}{n} .
\end{equation}
As a consequence of \eqref{eq:ims},
$$ 
\begin{aligned}
a(u_n) &= a(\psi_n u_n) + a(\widetilde \psi_n u_n)- \| \psi_n' u_n\|^2_{L^2(\HGraph)}- \| \widetilde{\psi}_n' u_n\|^2_{L^2(\HGraph)} \\
&=a(\psi_n u_n) + a(\widetilde \psi_n u_n) + O\left ( \frac{1}{n^2}\right ) \qquad (n\to \infty). 
\end{aligned}$$

We distinguish between two cases: (i) there exists a bounded subset $\mathcal K$ such that $\|u_n\|_{L^2(\HGraph_n \cap \mathcal K)} \not\to 0$, and (ii) $\|u_n\|_{L^2(\HGraph_n\cap \mathcal K)} \to 0$ for all bounded subsets $\mathcal K$.

In case (i), either $\|u\|_{L^2(\HGraph)}^2=1$, and 
from the weak lower semicontinuity $a(u) \leq \liminf_{n\to\infty} a(u_n)$ and \eqref{eq:help-james-please} we infer
\begin{equation}
    \infspec(\mathcal H)\le\mathcal R(u) \le \liminf_{n\to \infty} \mathcal R(u_n) = \infspec(\mathcal H),
\end{equation}
or else $0<\|u\|_{L^2} <1$ and, from \eqref{eq:ims} and the asymptotic minimizing property of the $u_n$,
$$ \min\left \{\frac{a(\psi_n u_n)}{\|\psi_n u_n\|_{L^2}^2}, \frac{a(\widetilde \psi_n u_n)}{\|\widetilde\psi_n u_n\|_{L^2}^2}\right \} \lesssim 
a(u_n) \lesssim \min\left \{\frac{a(\psi_n u_n)}{\|\psi_n u_n\|_{L^2}^2}, \frac{a(\widetilde \psi_n u_n)}{\|\widetilde \psi_n u_n\|_{L^2}^2}\right \} $$
(where ``$\lesssim$'' means that ``$\leq$'' holds between the respective limits). 
In particular,
$$
    \lim_{n\to \infty} \big[\infspec(\HGraph_n) -  \min\{\RQ(\psi_n u_n), \RQ(\widetilde\psi_n u_n)\}\big] = 0.
$$
Then due to weak lower semicontinuity, $a(u) \le \liminf_{n\to \infty} a(\psi_n u_n)$, and so
\begin{equation}
    \infspec(\HGraph) \leq \RQ(u) \leq \liminf_{n\to \infty} \RQ(\psi_n u_n). 
\end{equation}
On the other hand, since $\HGraph_n$ $\HGraph$ coincide outside a compact subset $\widetilde \psi_n u_n \in H^1_0(\HGraph_n)$ for sufficiently large $n$ and hence $\infspec(\HGraph) \leq  \RQ(\widetilde \psi_n u_n)$. We infer 
$$
    \infspec(\HGraph) \le \liminf_{n\to \infty} \RQ(u_n) \leq \liminf_{n\to\infty}\, \min\{\RQ(\psi_n u_n), \RQ(\widetilde\psi_n u_n)\} = \liminf_{n\to \infty} \infspec(\HGraph_n).
$$

In case (ii), consider $\psi_n, \widetilde \psi_n$ as in case (i). Similarly, we find a subsequence of $u_n$ still denoted by $u_n$ such that
$$ \int_{\HGraph} |\psi_n u_n|^2 \, \dx \to 0$$
and hence 
$$ \int_{\HGraph} |\widetilde \psi_n u_n|^2\, \mathrm dx \to 1.$$
Hence, as a consequence of \eqref{eq:ims} we have
$$ \RQ (\widetilde \psi_n u_n) \lesssim\infspec(\HGraph_n) \lesssim  \RQ (\widetilde \psi_nu_n) $$
and since $\HGraph_n, \HGraph$ coincide outside a compact subset 
$$\infspec(\HGraph) \le \liminf_{n\to \infty} \RQ(\widetilde \psi_n u_n) = \liminf_{n\to \infty} \infspec(\HGraph_n).  $$
\end{proof}

We next give natural conditions under which $\HGraph_n \to \HGraph$ in the sense of Mosco.

\begin{lemma}
\label{lem:hgraph-mosco-convergence}
Suppose that $\HGraph_n$, $\HGraph$ are subgraphs of $\Graph$ such that:
\begin{enumerate}
    \item[(a)] $\partial \HGraph_n$ and $\partial\HGraph$ coincide except for a finite set; and
    \item[(b)] either $\HGraph_n \subset \HGraph_{n+1} \subset \HGraph$ or $\HGraph \subset \HGraph_{n+1} \subset \HGraph_n$ for all $n$.
\end{enumerate}
Then:
\begin{enumerate}[(i)]
    \item If $\HGraph_n \to \HGraph$ topologically, then $H^1_0 (\HGraph_n)$ and $H^1_0 (\HGraph)$ satisfy condition (1) of Definition~\ref{def:mosco}.
    \item If $\HGraph_n$ and $\HGraph$ satisfy condition (1) of Definition~\ref{def:graph-convergence}, then $H^1_0 (\HGraph_n)$ and $H^1_0 (\HGraph)$ satisfy condition (2) of Definition~\ref{def:mosco}.
\end{enumerate}
In particular, if under conditions (a) and (b) $\HGraph_n \to \HGraph$ topologically, then $\HGraph_n \to \HGraph$ in the sense of Mosco.
\end{lemma}

\begin{proof}
For condition (1) of Mosco convergence there is nothing to prove if $\HGraph \subset \HGraph_{n+1} \subset \HGraph_n$ for all $n$. So suppose $\HGraph_n \subset \HGraph_{n+1} \subset \HGraph$, which in particular means by condition (1) of Definition~\ref{def:graph-convergence} that $\HGraph = \bigcup_{n\in\N} \HGraph_n$. Denote by $\VertexSet_\HGraph$ the finite boundary set of vertices of $\HGraph$ which are not also boundary points of $\HGraph_n$ for all $n \in \N$, and fix $f \in H^1_0 (\HGraph)$ arbitrary. Take a sequence of cutoff functions $\phi_m$ continuous, edgewise smooth such that $\phi_m, \phi_m'\in L^\infty(\HGraph)$, $0 \leq \phi_m \leq 1$, $\supp \phi_m \subset \HGraph$, and
\begin{displaymath}
	\phi_m(x) = \begin{cases} 1 \qquad &\text{if } \dist_\Graph (x,\VertexSet_\HGraph) \geq \tfrac{2}{m},\\
	0 \qquad &\text{if } \dist_\Graph(x,\VertexSet_\HGraph) \leq \tfrac{1}{m}.
	\end{cases}
\end{displaymath}
Then $f\phi_m \in H^1(\Graph)$ for all $m \in \N$ and, by our assumptions on $\HGraph_n$, for any $n \in \N$ we claim that $\supp \phi_m \subset \HGraph_n$ for all $m=m(n)$ large enough. Indeed, suppose for some $m$ (sufficiently large) and we have $\phi_m(x)>0$ for some $x$ which is outside $\HGraph_n$ for all $n \in \N$; then in fact $x$ must be an interior point of $\HGraph$, and $\phi_m(x)=1$ for all $m$ large enough. Since $\HGraph_n \to \HGraph$, there must exist a sequence of points $x_n \in \HGraph_n$ such that $x_n \to x$ in $\Graph$; since $x \not\in \HGraph_n$ there will exist some $v_n \in \partial\HGraph_n$ in between $x$ and $x_n$ (i.e. on a shortest path between the two points), so that also $v_n \to x$. But $v_n$ is clearly non-vanishing, since $\HGraph = \bigcup_n \HGraph_n$ and $x \in \interior \HGraph$; thus, by condition (2) of Definition~\ref{def:graph-convergence} and the uniqueness of limits we have $x \in \partial\HGraph$, a contradiction. This proves the claim.

Finally, given that $f\phi_m$ and $f$ disagree on a finite number of edges, a standard argument shows that $f\phi_m \to f$ in $H^1(\Graph)$.

For condition (2) there is nothing to prove if $\HGraph_n \subset \HGraph_{n+1} \subset \HGraph$. In the other case, taking a sequence $u_n \in H^1_0 (\HGraph_n)$ weakly convergent in $H^1(\Graph)$ to $u \in H^1(\Graph)$, we first observe that compactness of the imbedding $H^1 (\widetilde\Graph) \hookrightarrow C(\widetilde\Graph)$ on any \emph{bounded} subgraph $\widetilde\Graph \subset \Graph$ implies that $u_n \to u$ locally uniformly. We next observe that necessarily $\supp u \subset \HGraph = \bigcup_n \HGraph_n$, and that, for any $v \in \partial\HGraph$, by condition (1) of Definition~\ref{def:graph-convergence} there exist $v_n \in \partial\HGraph_n$ such that $v_n \to v$. Since $u_n(v_n)=0$ for all $n \in \N$, the local uniform convergence implies $u(v) = 0$. Since $v \in \partial\HGraph$ was arbitrary, we conclude that $u \in H^1_0 (\HGraph)$.
\end{proof}

\begin{corollary}
\label{cor:hgraph-mosco-convergence}
Suppose that $\HGraph_n$, $\HGraph$ are subgraphs of $\Graph$ which agree outside a compact set, i.e. we assume there exists some bounded set $\mathcal K\subset \Graph$ such that $\HGraph\setminus \mathcal K = \HGraph_n \setminus K$. Suppose that  $\HGraph_n$ and $\HGraph$ satisfy condition (1) of Definition~\ref{def:graph-convergence}, then 
\begin{displaymath}
\lambda(\HGraph) \ge \limsup_{n\to \infty} \lambda(\HGraph_n)
\end{displaymath}
If additionally $\HGraph, \HGraph_n$ satisfy condition (2) of Definition~\ref{def:graph-convergence}, that is, $\HGraph_n \to \HGraph$ topologically, then $\HGraph_n \to \HGraph$ in the sense of Mosco, and
\begin{equation}
\lambda(\HGraph_n) \to \lambda(\HGraph).
\end{equation}
\end{corollary}
\begin{proof}
Let 
\begin{displaymath}
\begin{aligned}
    \HGraph_n^+ &= \left \{ x\in \Graph : \dist_\Graph(x,\HGraph) \le \frac{1}{n} \right \},\\
    \HGraph_n^{-} &= \left \{ x\in \HGraph : \dist_\Graph(x, \Graph \setminus \HGraph) \le \frac{1}{n} \right \},
\end{aligned}
\end{displaymath} 
then due to condition (1) every subsequence admits a further subsequence still denoted by $\HGraph_n$, such that $\HGraph_n \subset \HGraph_n^+$. Now $\HGraph_n^+\to \HGraph$ in the topological sense; by hypothesis condition (a) of Lemma \ref{lem:hgraph-mosco-convergence} is satisfied, and (b) is satisfied by construction. Thus $\HGraph_n^+ \to \HGraph_n$ in the sense of Mosco. Due domain monotonicity (cf.\ Lemma~\ref{lem:monotonicity}(1)),
\begin{equation}\label{eq:lowersemicont-1}
    \lambda(\HGraph)=\lim_{n\to \infty} \lambda(\HGraph_n^+)\le \lim_{n\to \infty} \lambda(\HGraph_n)
\end{equation}
If additionally condition (2) is satisfied, then every subsequence admits a further subsequence still denoted by $\HGraph_n$, such that $\HGraph_n^- \subset \HGraph_n$. A similar argument as above shows that $\HGraph_n^- \to \HGraph$ in the sense of Mosco, and so, by domain monotonicity,
\begin{equation}\label{eq:cont-2}
    \lim_{n\to \infty}\lambda(\HGraph_n)\le \lim_{n\to \infty} \lambda(\HGraph_n^+) =\lambda(\HGraph). 
\end{equation}
From \eqref{eq:lowersemicont-1} and \eqref{eq:cont-2} we obtain $\lambda(\HGraph_n) \to \lambda(\HGraph)$.
\end{proof}

We now focus on the particular case of balls and exteriors of balls. That is, we imagine $\HGraph$ to be a fixed subgraph of $\Graph$, $x \in \Graph$ an arbitrary root vertex (not necessarily belonging to $\HGraph$), and look at subgraphs of the form $\HGraph \cap B_R(x)$ and $\HGraph \setminus B_R(x)$ as functions of $R$. In order to apply Lemma~\ref{lem:hgraph-mosco-convergence} and Corollary~\ref{cor:hgraph-mosco-convergence}, we need to show topological convergence. This is treated in the following lemma.

\begin{lemma}
\label{lem:hgraph-topological-convergence}
Fix any subgraph $\HGraph$ of $\Graph$ and any $x \in \Graph$, as well as $R_n \to R>0$ such that $\HGraph \cap B_R (x)$ and $\HGraph \setminus B_R (x)$ have postive volume. Then the symmetric difference of the sets $\partial (\HGraph \cap B_{R_n}(x))$ and $\partial (\HGraph \cap B_R (x))$ is finite, as is the symmetric difference of $\partial (\HGraph \setminus B_{R_n}(x))$ and $\partial (\HGraph \setminus B_R (x))$. Moreover,
\begin{displaymath}
    \HGraph \setminus B_{R_n}(x) \to \HGraph \setminus B_R (x)
\end{displaymath}
topologically as $n \to \infty$, while $\HGraph \cap B_{R_n}(x)$ and $\HGraph \cap B_R (x)$ satisfy condition (1) of Definition~\ref{def:graph-convergence}.
\end{lemma}

(Note that, since $\HGraph$ is a subgraph of the graph $\Graph$, which may have infinitely many edges, $\partial \HGraph$, and therefore also $\partial (\HGraph \cap B_R (x))$, need not be finite.)

\begin{proof}
First, for the finiteness of the symmetric differences: we treat the exterior case; the interior case is analogous. We may assume without loss of generality that $R \leq R_n < R + 1$ for all $n$. Then clearly $\partial (\HGraph \cap B_{R_n}(x))$ and $\partial (\HGraph \cap B_R(x))$ agree outside $B_{R+1}(x)$, where both boundaries coincide with $\partial \HGraph$ outside $B_{R+1}(x)$. Inside $B_{R+1} (x)$, our finiteness assumptions on $\Graph$, and the fact that $\HGraph$ is connected, imply that $\partial\HGraph$ is finite, as are the sets
\begin{displaymath}
\begin{aligned}
    \HGraph \cap \partial B_R (x) &= \{ v \in \HGraph: \dist_\Graph (v,x) = R\},\\ 
    \quad \HGraph \cap \partial B_{R_n} (x) &= \{ v \in \HGraph: \dist_\Graph (v,x) = R_n\}.
\end{aligned}
\end{displaymath}
This proves that the symmetric difference is always finite.

For the topological convergence, we first show that both sets satisfy Definition~\ref{def:graph-convergence}(1). Fix $R>0$; we will consider three cases:
\begin{enumerate}[(i)]
    \item $v \not\in \partial\HGraph$, so that $\dist_\Graph (v,x) = R$;
    \item $v \in \partial\HGraph$ but $\dist_\Graph (v,x) = R$;
    \item $v \in \partial\HGraph$ and $\dist_\Graph (v,x) \neq R$.
\end{enumerate}
Note in case (iii) that $v \in \partial (\HGraph \cap B_{r}(x))$ for all $r$ in a neighborhood of $R$, so that in this case we may simply take $v_n = v$.

For case (i), note that $B_\varepsilon (v) \subset \HGraph$ for $\varepsilon>0$ small enough; we assume without loss of generality that $B_\varepsilon (v)$ is a star, intersecting a finite number of edges of $\Graph$, each of which has $v$ as a vertex (in the case of loops we suppose that $\varepsilon$ is less than half the length of the shortest loop at $v$). Since $v$ is a boundary vertex between $B_R(x)$ and $\Graph \setminus B_R(x)$, there will be at least one ``incoming'' edge on which the distance to $x$ is $\leq R$, and at least one ``outgoing'' edge on which the distance is $\geq R$. Then it is immediate that for any $r \in (R-\varepsilon,R+\varepsilon)$ there exists a point $v_r \in B_\varepsilon(v)$ such that $\dist_\Graph (v_r, x) = R - \varepsilon$ and $\dist_\Graph (v_r,v) = \varepsilon$. This point is necessarily in $\HGraph \cap \partial B_{R-\varepsilon}(x)$; the claim of the lemma in this case follows.

We finally consider case (ii). Here $v \in \partial\HGraph \subset \HGraph$ will remain a boundary point of $\HGraph \cap B_r(x)$ for $r>R$ (as a boundary point of $\partial\HGraph$ contained in $\HGraph \cap B_r(x)$), and of $\HGraph \setminus B_r(x)$ for $r<R$. Consider the case of $\HGraph \cap B_{R_n} (x)$ where $R_n < R$, so that $v \not\in B_{R_n} (x)$. Since $v$ is not a degree zero vertex of $\HGraph \cap B_R (x)$, in any $\varepsilon$-neighborhood of $v$ there exist points $y \in \HGraph \cap B_R (x)$; in particular, for $\varepsilon > 0$ small enough we can find $y_\varepsilon \in \HGraph$ such that $\dist_\Graph (y_\varepsilon, x) < R$ and $\dist_\Graph (y_\varepsilon,v) = \varepsilon$. We claim that $y_\varepsilon$ lies on a shortest path from $x$ to $v$: indeed, any edge incident to $v$ for on which the distance to $x$ is always less than $R$ (i.e. an ``incoming'' edge) must have this property. In particular, if we choose $\varepsilon = R - R_n$ and set $v_n := y_{R-R_n}$, then $v_n \in \HGraph$ and $\dist_\Graph (v_n,x) = R_n$. Moreover, $v_n$ is clearly not a degree zero vertex for $n$ large enough. In the other case, $\HGraph \setminus B_{R_n}(x)$ for $R_n > R$, the argument is entirely analogous.

We now show condition (2) of topological convergence for $\HGraph \setminus B_R(x)$. Let $R_n \to R$ and take any sequence of non-vanishing boundary points $v_n \in \partial (\HGraph \setminus B_{R_n}(x))$. Suppose first that $\dist_\Graph (v_n, x) > R_n$ for infinitely many $n$; then for this subsequence we necessarily have $v_n \in \partial \HGraph$. Since by assumption the sequence is $v_n$ is bounded and $\partial \HGraph$ is finite in every bounded subset of $\Graph$, a subsubsequence of $v_n$ must necessarily be constant, hence convergent to some $v \in \partial\HGraph$ at distance $\geq R$ from $x$. The non-vanishing condition on the $v_n$ implies that $v$ is not an isolated point of $\HGraph \setminus B_R(x)$.

So we may suppose that $\dist_\Graph (v_n, x) = R_n$ for all $n$, and suppose without loss of generality that $R_n \leq R+1$ for all $n$, so that $v_n \in \HGraph \cap B_{R+1}(x)$. Again note that $\{y\in \HGraph: \dist_\Graph (y,x) = R_n\}$ is not only a finite set for all $n$, but its cardinality is uniformly bounded in $n$ (e.g. by twice the number of edges of $\Graph$ which intersect $B_{R+1} (x)$). Moreover, since $\HGraph \cap B_{R+1}(x)$ is compact, up to a subsequence $v_n$ converges to some $v \in \HGraph \cap B_{R+1}(x)$. Since $\dist_\Graph (v_n,x) = R_n \to R$, clearly $\dist_\Graph (v,x) = R$. Finally, we claim that $v \in \partial B_R (x)$. If $R_n > R$, then this follows immediately from $v_n \not\in B_R (x)$; if $R_n < R$, then we need to use that the $v_n$ are non-vanishing: indeed, it follows that for any $\varepsilon>0$,
\begin{displaymath}
    \big|B_\varepsilon (v) \cap \big(\HGraph \setminus B_{R_n}(x)\big)\big| \not\to 0.
\end{displaymath}
This shows that $v$ is not in the interior of $B_R(x)$, that is, $v \in \HGraph \cap \partial B_R(x)$, and thus equally $v \in \HGraph \setminus B_R(x) \subset \partial (\HGraph \setminus B_R(x))$.
\end{proof}

We can finally prove the continuity statements of Lemmata~\ref{lem:expanding-balls-at-0} and~\ref{lem:contracting-balls-at-infinity}.

\begin{lemma}
\label{lem:domain-continuity}
Under the assumptions of Lemma~\ref{lem:hgraph-topological-convergence}, $R \mapsto \infspec (\HGraph \setminus B_R(x))$ is a continuous mapping for all $R>0$ such that $\HGraph \setminus B_R(x) \neq \emptyset$, while $R \mapsto \infspec (\HGraph \cap B_R(x))$ is lower semicontinuous.
\end{lemma}

\begin{proof}
First note that conditions (a) and (b) of Lemma~\ref{lem:hgraph-mosco-convergence} are satisfied for sets of the form $\HGraph \cap B_{R_n}(x)$ and $\HGraph \setminus B_{R_n}(x)$ under the assumptions of Lemma~\ref{lem:hgraph-topological-convergence}, as long as either $R_n \uparrow R$, or $R_n \downarrow R$; moreover,  $\HGraph \cap B_R(x)$ and $\HGraph \cap B_{R_n}(x)$ clearly agree outside a compact subset of $\Graph$.

By Lemma~\ref{lem:hgraph-topological-convergence}, $\HGraph \setminus B_{R_n}(x) \to \HGraph \setminus B_R(x)$ topologically, while $\HGraph \cap B_R(x)$ and $\HGraph \cap B_{R_n}(x)$ satisfy condition (1) of Definition~\ref{def:graph-convergence}. We now apply Lemma~\ref{lem:hgraph-mosco-convergence}, if necessary dividing into subsequences for which $R_n \uparrow R$ and $R_n \downarrow R$, respectively, to obtain $\HGraph \setminus B_{R_n}(x) \to \HGraph \setminus B_R(x)$ in the sense of Mosco, while $H^1_0(\HGraph \cap B_R(x))$ and $H^1_0(\HGraph \cap B_{R_n}(x))$ satisfy condition (2) of Definition~\ref{def:mosco}.

Hence Lemma~\ref{lem:hgraph-ev-convergence} implies that $\infspec(\HGraph \setminus B_{R_n}(x)) \to \infspec(\HGraph \setminus B_R(x))$, while $\infspec (\HGraph \cap B_R(x)) \leq \liminf_{n\to\infty} \infspec (\HGraph \cap B_{R_n}(x))$, that is, $R \mapsto \infspec (\HGraph \cap B_R(x))$ is lower semicontinuous.
\end{proof}


\begin{thebibliography}{99}
\bibitem{AkPa16}
S.~Akduman and A.~Pankov, \emph{Schrödinger operators with locally integral potentials on infinite metric graphs}. Appl.\ Anal.\ \textbf{96} (2017), 2149--2161.

\bibitem{ABB18}
L.~Alon, R.~Band, and G.~Berkolaiko, \emph{Nodal statistics on quantum graphs}, Comm.\ Math.\ Phys.\ \textbf{362} (2018), 909--948.

\bibitem{ABB22}
L.~Alon, R.~Band, and G.~Berkolaiko, \emph{Universality of nodal count distribution in large metric graphs}, Experimental Math.\ (2022), 1--35.

\bibitem{BBRS12}
R.~Band, G.~Berkolaiko, H.~Raz, and U.~Smilansky, \emph{The number of nodal domains on quantum graphs as a stability index of graph partitions}, Comm.\ Math.\ Phys.\ \textbf{311} (2012), 815--838.

\bibitem{Berkolaiko08}
G.~Berkolaiko, \emph{A lower bound for nodal count on discrete and metric graphs}, Comm.\ Math.\ Phys.\ \textbf{278} (2008), 803--819.

\bibitem{BCCM22}
G.~Berkolaiko, Y.~Canzani, G.~Cox, and J.~L.~Marzuola, \emph{Spectral minimal partitions, nodal deficiency and the Dirichlet-to-Neumann map: the generic case}, preprint (2022), arXiv:2201.00773.

\bibitem{BCHPS22}
G.~Berkolaiko, G.~Cox, B.~Helffer, and M.~P.~Sundqvist, \emph{Computing nodal deficiency with a refined spectral flow}, preprint (2022), arXiv:2201.06667.

\bibitem{BeCoMa19}
G.~Berkolaiko, G.~Cox, and J.~L.~Marzuola, \emph{Nodal deficiency, spectral flow, and the Dirichlet-to-Neumann map}, Lett.\ Math.\ Phys.\ \textbf{109} (2019), 1611--1623.

\bibitem{BeKu13}
G.~Berkolaiko and P.~Kuchment, \emph{Introduction to quantum graphs}, volume 186 of \emph{Mathematical Surveys and Monographs}, American Mathematical Society, Providence, RI, 2013.

\bibitem{BNHe17}
V.~Bonnaillie-Noël and B.~Helffer, \emph{Nodal and spectral minimal partitions---the state of the art in 2016}. In A.~Henrot (ed.), \emph{Shape optimization and spectral theory}, pp.\ 353--397. De Gruyter Open, Warsaw, 2017.

\bibitem{CaFiNo17}
C.~Cacciapuoti, D.~Finco, and D.~Noja, \emph{Ground state and orbital stability for the NLS equation on a general starlike graph with potentials}, Nonlinearity \textbf{30} (2017), 3271--3303.

\bibitem{CoTeVe05}
M.~Conti, S.~Terracini, and G.~Verzini, \emph{On a class of optimal partition problems related to the Fuc\'ik spectrum and to the monotonicity formulae}, Calc.\ Var.\ Partial Differential Equations \textbf{22} (2005), 45--72.

\bibitem{DST20}
S.~Dovetta, E.~Serra, and P.~Tilli, \emph{NLS ground states on metric trees: existence results and open questions}, J.\ London Math.\ Soc.\ \textbf{102} (2020), 1223--1240.

\bibitem{EKMN18}
P.~Exner, A.~Kostenko, M.~Malamud, and H.~Neidhardt, \emph{Spectral theory of infinite quantum graphs}, Ann.\ Henri Poincar\'e \textbf{19} (2018), 3457--3510.

\bibitem{HiSi96}
P.~D.~Hislop and I.~M.~Sigal, \emph{Introduction to spectral theory}, volume 113 of \emph{Applied Mathematical Sciences}. Springer-Verlag, New York, NY, 1996.

\bibitem{Ho19}
M.~Hofmann, \emph{An existence theory for nonlinear equations on metric graphs via energy methods}, preprint (2019), arXiv:1909.07856.

\bibitem{Ho21}
M.~Hofmann, \emph{Spectral theory, clustering problems and differential equations on metric graphs}, PhD thesis, Universidade de Lisboa (Portugal), 2021. \url{http://hdl.handle.net/10451/50394}.

\bibitem{HoKe21}
M.~Hofmann and J.~B.~Kennedy, \emph{Interlacing and Friedlander-type inequalities for spectral minimal partitions of metric graphs}, Lett.\ Math.\ Phys.\ \textbf{111} (2021), 96.

\bibitem{HKMP21}
M.~Hofmann, J.~B.~Kennedy, D.~Mugnolo, and M.~Plümer, \emph{Asymptotics and estimates for spectral minimal partitions of metric graphs}, Integral Equations Oper.\ Theory \textbf{93} (2021), 26.

\bibitem{HKMP21pleijel}
M.~Hofmann, J.~B.~Kennedy, D.~Mugnolo, and M.~Plümer, \emph{On Pleijel's nodal domain theorem for quantum graphs}, Ann.\ Henri Poincar\'e \textbf{22} (2021), 3841--3870.

\bibitem{KKLM21}
J.~B.~Kennedy, P.~Kurasov, C.~L\'ena, and D.~Mugnolo, \emph{A theory of spectral partitions of metric graphs}, Calc.\ Var.\ Partial Differential Equations \textbf{60} (2021), 61.

\bibitem{KoMuNi19}
A.~Kostenko, D.~Mugnolo, and N.~Nicolussi, \emph{Self-adjoint and Markovian extensions of infinite quantum graphs}, J.\ London Math.\ Soc.\ \textbf{105} (2022), 1262--1313.

\bibitem{KoNi19}
A.~Kostenko and N.~Nicolussi, \emph{Spectral estimates for infinite quantum graphs}, Calc.\ Var.\ Partial Differential Equations \textbf{58} (2019), 15.

\bibitem{Ku19}
P.~Kurasov, \emph{On the ground state for quantum graphs}, Lett.\ Math.\ Phys.\ \textbf{109} (2019), 2491--2512.

\bibitem{Mu19}
D.~Mugnolo, \emph{What is actually a metric graph?}, preprint (201), arXiv:1912.07549.

\bibitem{Ni87}
S.~Nicaise, \emph{Spectre des r\'eseaux topologique finis}, Bull.\ Sci.\ Math.\ (2) \textbf{111} (1987), 401--413.

\bibitem{Pe60}
A.~Persson, \emph{Bounds for the discrete part of the spectrum of a semi-bounded Schrödinger operator}, Math.\ Scand.\ \textbf{8} (1960), 143--153.

\bibitem{Si85}
B.~Simon, \emph{Some aspects of the theory of Schrödinger operators}. In \emph{Schrödinger operators (Como, 1984)}, volume 1159 of \emph{Lecture Notes in Math.}, pp.\ 177--203. Springer, Berlin, 1985.

\bibitem{SoSo02}
A.~V.~Sobolev and M.~Solomyak, \emph{Schrödinger operators on homogeneous metric trees: spectrum in gaps}, Rev.\ Math.\ Phys.\ \textbf{14} (2002), 421--467.

\bibitem{So04}
M.~Solomyak, \emph{On the spectrum of the Laplacian on regular metric trees}, Waves Random Media \textbf{14} (2004), 155--171.

\bibitem{Te14}
G.~Teschl, \emph{Mathematical methods in quantum mechanics}, volume 157 of \emph{Graduate Studies in Mathematics}. 2nd edition. American Mathematical Society, Providence, RI, 2014.
\end{thebibliography}
\end{document}